\documentclass[11pt]{amsart}

\usepackage[utf8]{inputenc}
\usepackage[T1]{fontenc}

\usepackage{amsmath,amssymb,amsfonts,textcomp,amsthm,xifthen,graphicx,color,pgfplots}
\usepackage{ifthen}

\usepackage{enumerate}
\usepackage{enumitem}

\usepackage{fullpage}

\usepackage[utf8]{inputenc}

\usepackage[pdftex,
            pdfauthor={Thomas F\"uhrer},
            pdftitle={Ultraweak formulation of linear PDEs in nondivergence form and DPG approximation},
            ]{hyperref}

\usepackage{pgfplots}
\usepgfplotslibrary{colorbrewer}
\pgfplotsset{compat=newest}
\usepgfplotslibrary{external}
\usepackage{float} 
\usepackage{array}

\usepackage{multirow}

\usepackage{etoolbox}
\newcounter{rowcntr}[table]
\renewcommand{\therowcntr}{\alph{rowcntr})}

\newcolumntype{N}{>{\refstepcounter{rowcntr}\therowcntr}c}

\AtBeginEnvironment{tabular}{\setcounter{rowcntr}{0}}

\newtheorem{theorem}{Theorem}
\newtheorem{lemma}[theorem]{Lemma}
\newtheorem{corollary}[theorem]{Corollary}
\newtheorem{proposition}[theorem]{Proposition}
\newtheorem{remark}[theorem]{Remark}

\newcommand\sign{\operatorname{sign}}

\newcommand\etaDPG{\eta_\mathrm{DPG}}
\newcommand\etaLSQ{\eta_\mathrm{LS}}

\newcommand\osc{\mathrm{osc}}

\def\evmin{\lambda_{\rm min}}
\def\evmax{\lambda_{\rm max}}

\newcommand{\ZZ}{\boldsymbol{Z}}

\newcommand{\ip}[2]{(#1\hspace*{.5mm},#2)}
\newcommand{\dual}[2]{\langle#1\hspace*{.5mm},#2\rangle}

\newcommand{\norm}[3][]{#1\|#2#1\|_{#3}}

\def\div{{\rm div\,}}
\def\Div{{\mathbf{Div}\,}}

\newcommand{\Hdivset}[1]{\boldsymbol{H}(\div;#1)}
\newcommand{\HdivDivset}[1]{\boldsymbol{H}(\div\Div;#1)}
\newcommand{\HdivDivNset}[1]{\boldsymbol{H}_0(\div\Div;#1)}

\newcommand{\set}[2]{\big\{#1\,:\,#2\big\}}

\newcommand{\R}{\ensuremath{\mathbb{R}}}
\newcommand{\N}{\ensuremath{\mathbb{N}}}

\newcommand{\LL}{\ensuremath{\boldsymbol{L}}}

\newcommand{\vv}{\ensuremath{\boldsymbol{v}}}
\newcommand{\ww}{\ensuremath{\boldsymbol{w}}}
\newcommand{\TT}{\ensuremath{\mathcal{T}}}

\newcommand{\cS}{\ensuremath{\mathcal{S}}}

\newcommand{\PP}{\ensuremath{\mathcal{P}}}

\newcommand{\EE}{\ensuremath{\mathcal{E}}}

\newcommand{\normal}{\ensuremath{{\boldsymbol{n}}}}
\newcommand{\tangential}{\ensuremath{{\boldsymbol{t}}}}

\newcommand{\VV}{\ensuremath{\mathcal{V}}}

\newcommand{\GG}{\ensuremath{\boldsymbol{G}}}

\newcommand{\CC}{\ensuremath{\mathcal{C}}}

\newcommand\dDiv{\operatorname{div\mathbf{Div}}}

\newcommand\DD{\mathcal{D}}

\newcommand{\uu}{\boldsymbol{u}}

\newcounter{constantsnumber}
\def\setc#1{
  \ifthenelse{\equal{#1}{poinc}}{C_{\rm edge}}{ 
   \refstepcounter{constantsnumber}
   \label{const#1}C_{\theconstantsnumber}}}

\def\c#1{
  \ifthenelse{\equal{#1}{poinc}}{C_{\rm edge}}{ 
    C_{\ref{const#1}}}}

\newcommand{\MM}{\boldsymbol{M}}
\newcommand{\QQ}{\boldsymbol{Q}}
\newcommand{\tr}{\operatorname{tr}}
\newcommand{\trdDiv}{\tr^\mathrm{dDiv}}
\newcommand{\trHtwo}{\tr^2}

\begin{document}

\title{Ultraweak formulation of linear PDEs in \\nondivergence form and DPG approximation}
\date{\today}

\author{Thomas F\"{u}hrer}
\address{Facultad de Matem\'{a}ticas, Pontificia Universidad Cat\'{o}lica de Chile, Santiago, Chile}
\email{tofuhrer@mat.uc.cl}

\thanks{{\bf Acknowledgment.} 
This work was supported by FONDECYT project 11170050.}

\begin{abstract}
We develop and analyze an ultraweak formulation of linear PDEs in nondivergence form where the coefficients satisfy the
Cordes condition.
Based on the ultraweak formulation we propose discontinuous Petrov--Galerkin (DPG) methods.
We investigate Fortin operators for the fully discrete schemes and provide a posteriori estimators for the methods under consideration.
Numerical experiments are presented in the case of uniform and adaptive mesh-refinement.
\end{abstract}
\maketitle

\section{Introduction}
Let $\Omega \subset \R^d$ ($d=2,3$) be a bounded convex polytopal domain with boundary $\Gamma:=\partial\Omega$.
We consider the problem of finding the solution $u$ to the PDE
\begin{align}\label{eq:model}
\begin{split}
  \mathcal{L}u &= f \quad\text{in } \Omega, \\
  u|_\Gamma &= 0,
\end{split}
\end{align}
where 
\begin{subequations}
\begin{align}
  \mathcal{L}u(x) := \sum_{j,k=1}^d A_{jk}(x)\frac{\partial^2 u}{\partial x_j\partial x_k}(x)
\end{align}
and $A\colon \Omega \to \R^{d\times d}_\mathrm{sym}$ satisfies
\begin{align}\label{eq:cond:ell}
  0 < \operatorname{ess\inf}_{x\in\Omega}\evmin(x) \leq \operatorname{ess\sup}_{x\in\Omega} \evmax(x) < \infty.
\end{align}
and the Cordes condition: There exists $0<\varepsilon\leq 1$ such that
\begin{align}\label{eq:cond:cord}
  \frac{\norm{A}F^2}{(\tr A)^2} := \frac{\sum_{j,k=1}^d A_{jk}^2}{\big(\sum_{j=1}^d A_{jj}\big)^2} \leq \frac{1}{d-1+\varepsilon}
  \quad\text{a.e. in }\Omega.
\end{align}
\end{subequations}
It is known that the problem admits a unique strong solution $u\in H^2(\Omega)\cap H_0^1(\Omega)$, see
e.g.,~\cite{SmearsSueli13} and references therein.
Let us also note that for $d=2$ condition~\eqref{eq:cond:ell} implies~\eqref{eq:cond:cord}.

In recent years various numerical methods for problem~\eqref{eq:model} have been proposed.
One of the first works dealing with finite element schemes is by Smears \& S\"{u}li who defined and analyzed a
discontinuous Galerkin finite element method (DG-FEM)~\cite{SmearsSueli13}.
Gallistl defined and analyzed a non-symmetric mixed FEM and a least-squares finite element method (LS-FEM)
in~\cite{Gallistl17NONDIV}.
The mixed formulation is based on the theory of stable splittings of polyharmonic equations developed
in~\cite{GallistlPOLY}, where the problem to solve decouples into a Stokes-type problem and
subproblems allowing $H^1$ conforming discretizations.
Extensions to nonlinear problems of the latter works are found in~\cite{SmearsSueli14,SmearsSueli16,GallistlSueli19}.
Other contributions include~\cite{FengHenningsNeilan17,FengNeilanSchnake18,LakkisPryer11,weakGalerkin}.

As pointed out in the articles cited above, efficient numerical schemes to approximate solutions of~\eqref{eq:model}
are necessary since problems of type~\eqref{eq:model} appear when linearising fully nonlinear problems like the
Hamilton--Jacobi--Bellman equations. We refer the interested reader to~\cite{SmearsSueli14}.

In the present work we relax~\eqref{eq:model}, i.e., $A:D^2u = f$, by introducing an auxiliary variable $\MM = D^2u$.
The problem is thus recast to the system
\begin{align*}
  A:\MM &= f,\\
  \MM-D^2 u &= 0.
\end{align*}
The second equation of this system will be considered in a very weak (\emph{ultraweak}) sense by testing with
discontinuous functions and then shifting the derivatives applied to $u$ to the test functions.
This approach needs an additional trace variable that carries the continuity information of the solution.
Let us point out some properties of our approach: 
The symmetric matrix $\MM$ has coefficients in $L^2(\Omega)$ and we can therefore simply approximate them with
discontinuous functions, e.g., piecewise polynomials. 
Since in the ultraweak setting no derivatives are applied to $u$ we can approximate it also with discontinuous functions.
For the approximation of the trace variable we use traces of $H^2(\Omega)$ functions.
For the numerical examples ($d=2$) we use traces of the reduced Hsiegh--Clough--Tocher (rHCT) elements.
We stress that, following~\cite{KLove1}, for the analysis and implementation we only need to know the traces of such
functions, which in the case of rHCT elements are polynomials, but there is no need to know their 
presentation in the interior.
This is a particular advantage compared to other methods which use $H^2(\Omega)$ conforming discretization spaces, e.g.,
the LS-FEM from~\cite{Gallistl17NONDIV}. 

Applying the DPG methodology of Demkowicz \& Gopalakrishnan~\cite{partI,partII,DPGoverview} to the ultraweak formulation
gives us automatically: Algebraic systems that are symmetric and positive definite and local error indicators to steer adaptive mesh-refinement.

\textbf{Outline.}
Section~\ref{sec:varforms} introduces the notation, functional analytical setting, and the definition of the DPG methods
together with the main results.
The proofs are postponed to Section~\ref{sec:proofs}.
In Section~\ref{sec:aposteriori} we analyze error estimators. Fortin operators for the fully discrete schemes of the proposed methods are
considered in Section~\ref{sec:fullydiscrete}.
Finally, numerical experiments are given in Section~\ref{sec:examples}.

\section{Variational formulations and DPG methods}\label{sec:varforms}

\subsection{Notation}
For any subdomain $\omega \subseteq \Omega$ we use the common notation $L^2(\omega)$ for square integrable functions and $H^k(\omega)$
for Sobolev spaces of order $k$. 
Particularly, $H_0^1(\omega)$ denotes the space with vanishing traces and $H_0^2(\omega)$
denotes the space of $H^2(\omega)$ functions with vanishing traces and vanishing traces of the gradient.
The $L^2(\omega)$ inner product is denoted by $\ip\cdot\cdot_\omega$ and if $\omega=\Omega$ we skip the index, i.e.,
$\ip\cdot\cdot_\Omega := \ip\cdot\cdot$.
The norm induced by the $L^2(\omega)$ inner product is denoted with $\norm\cdot{\omega}$ and if $\omega=\Omega$ we skip the
index as before, i.e., $\norm\cdot\Omega = \norm\cdot{}$.
We will also work with
\begin{align*}
  \LL_\mathrm{sym}^2(\omega) = \set{\QQ\in L^2(\omega)^{d\times d}}{\QQ = \QQ^\top},
\end{align*}
i.e. symmetric matrices with $L^2(\omega)$ coefficients.
The inner product and norm are denoted with the same symbols as in the scalar case, for instance,
\begin{align*}
  \ip{\MM}{\QQ}_\omega = \int_\omega \MM : \QQ \,dx,
\end{align*}
where the colon operator stands for the Frobenius product.

For a function $v\in H^2(\omega)$ the Hessian is denoted by $D^2 v \in \LL_\mathrm{sym}^2(\omega)$. 
The norm of a Sobolev function is given by $\norm{v}{H^k(\omega)}^2 = \norm{v}\omega^2 + \norm{D^k v}{\omega}^2$.

We will also use the
(formally) $L^2$ adjoint operator of $D^2$ denoted by $\dDiv$, i.e., the double iterated divergence, where $\div$ denotes the standard divergence
operator and $\Div$ denotes the row-wise divergence operator.
We define the space $\HdivDivset\omega$ as the completion of $\DD(\overline\omega)^{d\times d}\cap \LL_\mathrm{sym}^2(\omega)$ with respect to the norm
\begin{align*}
  \norm{\cdot}{\HdivDivset\omega} = \sqrt{\norm{\cdot}{\omega}^2 + \norm{\div\Div(\cdot)}\omega^2}.
\end{align*}

Let $\TT$ denote a partition of the domain $\Omega$ into non-intersecting open Lipschitz subdomains $T$ with positive measure, i.e.,
$\bigcup_{T\in\TT} \overline T = \overline\Omega$ and $|T|>0$.
We define the product spaces
\begin{align*}
  H^k(\TT) &:= \prod_{T\in\TT} H^k(T), \\
  \HdivDivset\TT &:= \prod_{T\in\TT} \HdivDivset{T}.
\end{align*}
Clearly, these spaces can be identified as subspaces of $L^2(\Omega)$ and $\LL_\mathrm{sym}^2(\Omega)$, respectively.
The norms are given by
\begin{align*}
  \norm{\cdot}{H^k(\TT)} &= \sqrt{\sum_{T\in\TT} \norm{\cdot}{H^k(T)}^2}, \\
  \norm{\cdot}{\HdivDivset\TT} &= \sqrt{\sum_{T\in\TT} \norm{\cdot}{\HdivDivset{T}}^2}.
\end{align*}

For the definition of the bilinear form associated to the ultraweak formulation we will use the short notation
\begin{align*}
  \ip{u}{\div\Div\QQ}_\TT = \sum_{T\in\TT} \ip{u}{\div\Div\QQ}_T \quad\forall u\in L^2(\Omega),\, \QQ\in\HdivDivset\TT.
\end{align*}

Throughout, if not stated otherwise, $C$ (probably with an additional index) denotes a generic constant that depends on
$\Omega$ and the coefficient matrix $A$. We write $a\lesssim b$ if $a\leq C\,b$ and $a\simeq b$ if $a\lesssim b$ and
$b\lesssim a$.

\subsection{Strong form}
Let us note that by testing the PDE in~\eqref{eq:model} with $\gamma \Delta v$ where $v\in X:= H^2(\Omega)\cap H_0^1(\Omega)$ 
and $\gamma(x) := \tr A(x) / \norm{A(x)}F^2$ we obtain the
variational formulation 
\begin{align}\label{eq:varform:strong}
  u \in X : \quad
  \ip{A:D^2u}{\gamma \Delta v} = \ip{f}{\gamma \Delta v} \quad\forall v\in X.
\end{align}
It is shown in~\cite{SmearsSueli13} that this problem admits a unique solution.
More precisely, the bilinear form defined by the left-hand side is bounded and coercive,
\begin{align*}
  |\ip{A:D^2u}{\gamma\Delta v}|&\lesssim \norm{u}{H^2(\Omega)}\norm{v}{H^2(\Omega)}, \\
  \norm{u}{H^2(\Omega)}^2 &\lesssim \ip{A:D^2 u}{\gamma\Delta u}
\end{align*}
for all $u,v\in X$, see~\cite[Proof of Theorem~3]{SmearsSueli13}.
\begin{proposition}
  Let $f\in L^2(\Omega)$. Problem~\eqref{eq:varform:strong} admits a unique solution $u \in X$ and satisfies
  \begin{align*}
    \norm{u}{H^2(\Omega)} \leq C \norm{f}{}.
  \end{align*}
\end{proposition}

Moreover, observe that $\Delta X = L^2(\Omega)$ since $\Omega$ is convex, hence, 
solutions of~\eqref{eq:varform:strong} are the (strong) solution of~\eqref{eq:model} ($\gamma$ is a positive, essentially
bounded weight function).

We recall that
\begin{align*}
  \norm{\Delta v}{} \simeq \norm{D^2 v}{} \simeq \norm{v}{H^2(\Omega)} \quad\forall v\in X.
\end{align*}
Proofs for the estimate $\norm{D^2 v}{} \lesssim \norm{\Delta v}{}$ are found in~\cite{grisvard}.

The properties of the bilinear form in~\eqref{eq:varform:strong} also imply the following result which we use
in our analysis below:
\begin{lemma}\label{lem:blfa}
  The bilinear form $a(u,w):=\ip{A:D^2u}{w}$ is bounded and satisfies the $\inf$--$\sup$ conditions on $X\times
  L^2(\Omega)$.
\end{lemma}
\begin{proof}
  Boundedness is straightforward to prove. For the $\inf$--$\sup$ conditions let $u\in X$ be given and choose $w =
  \gamma\Delta u \in L^2(\Omega)$. Then, using $\ip{A:D^2u}{\gamma \Delta u} \gtrsim \norm{D^2 u}{}^2
  \simeq \norm{\Delta u}{}^2$ we get that
  \begin{align*}
    \sup_{0\neq \widetilde w\in L^2(\Omega)} \frac{\ip{A:D^2u}{\widetilde w}}{\norm{\widetilde w}{}}
    \geq \frac{\ip{A:D^2u}{w}}{\norm{w}{}} \gtrsim \frac{\norm{D^2 u}{}^2}{\norm{\Delta u}{}} \simeq \norm{D^2 u}{}.
  \end{align*}
  For the other condition let $w\in L^2(\Omega)$ be given and let $u\in X$ denote the solution of
  problem~\eqref{eq:model} with right-hand side $f = w$. This shows that
  \begin{align*}
    \sup_{0\neq \widetilde u\in X} \ip{A:D^2 \widetilde u}{w} \geq \ip{A:D^2 u}{w} = \norm{w}{}^2
  \end{align*}
  which concludes the proof.
\end{proof}


%

\subsection{Traces}
Similar as in the case of normal traces for $\Hdivset\omega$ elements we define the trace associated to the space
$\HdivDivset\omega$ via integration by parts. 
A detailed analysis is found in our work on the Kirchhoff--Love plate problem~\cite{KLove1}.
The operator $\trdDiv_\omega\colon \HdivDivset\omega \to H^2(\omega)'$, where $H^2(\omega)'$ is the dual space of
$H^2(\omega)$, is for any $\MM\in\HdivDivset\omega$ defined through
\begin{align*}
  (\trdDiv_\omega\MM)(v) := \ip{\div\Div\MM}{v}_\omega - \ip{\MM}{D^2 v}_\omega \quad\forall v\in H^2(\omega).
\end{align*}
As discussed in~\cite{KLove1} the functional $\trdDiv_\omega\MM$ is only supported on the boundary $\partial\omega$,
i.e., $\trdDiv_\omega\MM (v) = 0$ for all $v\in H_0^2(\omega)$.
Moreover, if $\MM$ and $v$ are regular enough functions then integration by parts shows that the right-hand side reduces
to boundary integrals. It is therefore natural to introduce the notation
\begin{align*}
  \dual{\trdDiv_\omega\MM}v_{\partial\omega} := (\trdDiv_\omega\MM)(v)
\end{align*}
which we will use throughout this work.

We also make use of the trace operator restricted to testing with functions in $X = H^2(\Omega)\cap H_0^1(\Omega)$.
Formally, integration by parts gives for $v$ with $v|_\Gamma =0$
\begin{align*}
  \ip{\div\Div\MM}{v} - \ip{\MM}{D^2 v} = -\dual{\normal\cdot\MM\normal}{\partial_\normal v}_\Gamma .
\end{align*}
This motivates the definition of the normal-normal trace $\normal\cdot(\cdot)\normal\colon \HdivDivset\Omega \to X'$,
\begin{align*}
  \dual{\normal\cdot\MM\normal}{\partial_\normal v}_\Gamma := -\dual{\trdDiv_\Omega\MM}{v}_\Gamma \quad\forall v\in X.
\end{align*}
Moreover, this also gives rise to the definition of the space
\begin{align*}
  \HdivDivNset\Omega := \set{\MM\in \HdivDivset\Omega}{\normal\cdot\MM\normal = 0}
\end{align*}
which is a closed subspace of $\HdivDivset\Omega$.

In the same manner we define a trace operator for $H^2(\omega)$ functions, i.e., for $u\in H^2(\omega)$,
\begin{align*}
  \dual{\trHtwo_\omega u}{\QQ}_{\partial \omega} := \ip{\div\Div\QQ}{u}_\omega - \ip{\QQ}{D^2 u}_\omega
  \quad\forall \QQ\in\HdivDivset\omega.
\end{align*}
Again we notice that a more detailed discussion on this trace operator is found in~\cite{KLove1}.

We define collective versions of the trace operators introduced above:
\begin{align*}
  \trHtwo_\TT &: H^2(\Omega) \to \HdivDivset\TT', \quad\text{and } \\
  \trdDiv_\TT &: \HdivDivset\Omega \to H^2(\TT)',
\end{align*}
where
\begin{align*}
  (\trHtwo_\TT u)(\QQ) &:= \dual{\trHtwo_\TT u}\QQ_{\cS} := \sum_{T\in\TT} \dual{\trHtwo_T u}{\QQ}_{\partial T}
  \quad\text{and}\\
  (\trdDiv_\TT \MM)(v) &:= \dual{\trdDiv_\TT \MM}v_{\cS} := \sum_{T\in\TT} \dual{\trdDiv_T \MM}{v}_{\partial T}.
\end{align*}

Our ultraweak formulation relies on the traces of functions in $X$,
\begin{align*}
  \widehat U := \trHtwo_\TT (X)
\end{align*}
equipped with the natural trace norm, i.e., the \emph{minimum energy extension} norm given by
\begin{align*}
  \norm{\widehat\uu}{3/2,1/2,\cS} := \inf \set{\norm{u}{H^2(\Omega)}}{\trHtwo_\TT u = \widehat\uu}.
\end{align*}
Following our own work we have
\begin{proposition}[{\cite[Proposition~3.9]{KLove1}}]\label{prop:traces}
  For $\widehat\uu\in \widehat U$ it holds the identity
  \begin{align*}
    \norm{\widehat \uu}{3/2,1/2,\cS} = \sup_{0\neq \QQ\in\HdivDivset\TT} 
    \frac{\dual{\widehat \uu}{\QQ}_{\cS}}{\norm{\QQ}{\HdivDivset\TT}}.
  \end{align*}
\end{proposition}
We note that~\cite[Proposition~3.9]{KLove1} has been stated and proved for traces of $H_0^2(\Omega)$ functions but
equally applies to the present case with traces of $H^2(\Omega)\cap H_0^1(\Omega)$ functions.
See also the recent work~\cite[Section~3.3]{ReissnerMindlinDPG}.

\subsection{Ultraweak formulation}
First, we rewrite problem~\eqref{eq:model} as
\begin{subequations}\label{eq:reformulation}
\begin{align}
  A:\MM &= f, \\
  \MM-D^2 u &= 0, \\
  u|_\Gamma &= 0.
\end{align}
\end{subequations}

Then, we test the first equation with some $v\in L^2(\Omega)$ and the second with $\QQ\in \HdivDivset\TT$. This (formally)
leads to
\begin{align*}
  \sum_{T\in\TT} \left( \ip{A:\MM}v_T + \ip{\MM}{\QQ}_T - \ip{u}{\dDiv\QQ}_T + \dual{\widehat\uu}\QQ_{\partial T} \right) =
  \sum_{T\in\TT}\ip{f}v_T,
\end{align*}
where $\widehat\uu = \trHtwo_\TT u$.

For the functional analytic setting of this formulation we will work in the spaces
\begin{align*}
  U &:= L^2(\Omega) \times \LL_\mathrm{sym}^2(\Omega) \times \widehat U, \\
  V &:= L^2(\Omega) \times \HdivDivset\TT,
\end{align*}
equipped with the norms
\begin{align*}
  \norm{\uu}U^2 &:= \norm{u}{}^2 + \norm{\MM}{}^2 + \norm{\widehat\uu}{3/2,1/2,\cS}^2, \\
  \norm{\vv}V^2 &:= \norm{v}{}^2 + \norm{\QQ}{\HdivDivset\TT}^2
\end{align*}
for $\uu=(u,\MM,\widehat\uu)\in U$, $\vv = (v,\QQ)\in V$.
The bilinear form and the right-hand side functional corresponding to the ultraweak formulation then read
\begin{align}
  b(\uu,\vv) &:= \ip{u}{-\dDiv\QQ}_\TT + \ip{\MM}{Av+\QQ} + \dual{\widehat\uu}{\QQ}_\cS, \\
  F(\vv) &:= \ip{f}v,
\end{align}
for all $\uu = (u,\MM,\widehat\uu) \in U$, $\vv=(v,\QQ)\in V$.

\begin{theorem}\label{thm:dpg}
  Let $f\in L^2(\Omega)$. The problem 
  \begin{align}\label{eq:ultraweak}
    \uu \in U : \quad
    b(\uu,\vv) = F(\vv) \quad\forall \vv\in V
  \end{align}
  admits a unique solution $\uu^\star\in U$ and it satisfies 
  \begin{align*}
    \norm{\uu^\star}U \leq C \norm{F}{V'} = C\norm{f}{}.
  \end{align*}
\end{theorem}
A proof is presented in Section~\ref{sec:proofs}. We remark that in Section~\ref{sec:proofs} we show that the bounded
bilinear form $b(\cdot,\cdot)$ satisfies the $\inf$--$\sup$ conditions. 
Thus, the latter result is true for general data $F\in V'$.

\begin{proposition}
  Problems~\eqref{eq:model} and~\eqref{eq:ultraweak} are equivalent in the following sense:
  If $u\in X$ solves~\eqref{eq:model}, then $\uu:=(u,D^2 u,\trHtwo_\TT u)\in U$ solves~\eqref{eq:ultraweak}.
  If $\uu=(u,\MM,\widehat\uu)\in U$ solves~\eqref{eq:ultraweak}, then $u\in X$ solves~\eqref{eq:model}.
\end{proposition}
\begin{proof}
  Let $u\in X$ denote the solution of~\eqref{eq:model}. 
  By construction $\uu=(u,\MM,\widehat\uu) = (u,D^2u,\trHtwo_\TT u)\in U$ satisfies the ultraweak
  formulation~\eqref{eq:ultraweak}.

  Let $\uu = (u,\MM,\widehat\uu)\in U$ solve the ultraweak formulation. Let $T\in \TT$ be given. We
  test~\eqref{eq:ultraweak} with $\vv = (0,\QQ_T)$ where $\QQ_T$ is smooth and has compact support in $T$. 
  It follows that $\dual{\widehat\uu}{\QQ_T}_\cS = \dual{\widehat\uu}{\QQ_T}_{\partial T} = 0$. Therefore, from
  $b(\uu,\vv) = 0$ we infer that
  \begin{align*}
    \ip{\MM}{\QQ_T}_T-\ip{u}{\dDiv\QQ_T}_T = 0.
  \end{align*}
  This means that $D^2 u|_T = \MM|_T$ in the distributional sense
  and since $\MM\in \LL_\mathrm{sym}^2(\Omega)$ we have that $u|_T\in H^2(T)$ for all $T\in\TT$.
  By testing with $\vv=(0,\QQ_T)$ where $\QQ_T|_T\in \HdivDivset{T}$ and $\QQ_T|_{T'}=0$ for $T\neq T'\in\TT$ we obtain
  with the definition of the trace operator that $\trHtwo_T u = \widehat\uu|_{\partial T}$.
  In other words we have shown that $u$ is elementwise an $H^2$ function and its elementwise traces equal to
  $\widehat\uu$, thus, with standard arguments we conclude that $u\in H^2(\Omega)$ and $D^2u = \MM$. Moreover, $u\in X$.
  Finally, testing with $\vv = (v,0)$ where $v\in L^2(\Omega)$ in~\eqref{eq:ultraweak} and using that $D^2 u=\MM$ this shows that $u$ solves the
  strong formulation~\eqref{eq:model}.
\end{proof}

\subsection{DPG method}\label{sec:dpgmethod}
The DPG method~\cite{partI,partII} selects optimal test functions which are computed using the trial-to-test operator
$\Theta\colon U\to V$ defined via the relation
\begin{align*}
  \ip{\Theta \uu}{\vv}_V = b(\uu,\vv) \quad\forall \vv\in V,
\end{align*}
where $\ip\cdot\cdot_V$ denotes the inner product in $V$ that induces the norm $\norm\cdot{V}$.
Furthermore, we have that
\begin{align*}
  \sup_{0\neq \vv\in V} \frac{b(\uu,\vv)}{\norm{\vv}{V}} = \sup_{0\neq \vv\in V} \frac{\ip{\Theta\uu}\vv_V}{\norm{\vv}{V}} =
  \norm{\Theta\uu}{V} = b(\uu,\Theta\uu)^{1/2} \quad\forall
  \uu\in U.
\end{align*}
Particularly, if the left-hand side is bounded below by $\norm{\uu}U$ then the bilinear form $b(\cdot,\Theta(\cdot))$ is
coercive on $U$.

We stress that in the proof of Theorem~\ref{thm:dpg} we show that $b(\cdot,\cdot)$ satisfies the $\inf$--$\sup$
conditions. Together with boundedness of $b(\cdot,\cdot)$ on $U\times V$, the observations from above and the Lax--Milgram lemma one concludes:
\begin{theorem}\label{thm:dpg:discrete}
  Let $U_h\subset U$ be some finite dimensional space. The problem
  \begin{align}\label{eq:dpg:discrete}
    \uu_h \in U_h : \quad 
    b(\uu_h,\Theta\ww_h) = F(\Theta\ww_h) \quad\forall \ww_h\in U_h
  \end{align}
  admits a unique solution.

  Let $\uu\in U$ denote the solution of~\eqref{eq:ultraweak} and let $\uu_h\in U_h$ denote the solution
  of~\eqref{eq:dpg:discrete}, then
  \begin{align}\label{eq:quasiopt}
    \norm{\uu-\uu_h}U \leq C_\mathrm{qopt} \min_{\ww_h\in U_h} \norm{\uu-\ww_h}U.
  \end{align}
\end{theorem}

We note that for a practical method we also have to take into account approximations of the optimal test functions.
For many DPG methods this is usually done by choosing a discrete test space that allows the
existence of a Fortin operator.
To that end we make the general assumption that there exists a finite-dimensional subspace $V_h\subset V$, and an
operator $\Pi_F : V \to V_h$ with bounded operator norm, 
\begin{subequations}\label{eq:fortin}
\begin{align}
  \norm{\Pi_F \vv}{V} \leq C_F \norm{\vv}{V} \quad\forall \vv\in V,
\end{align}
and the Fortin property
\begin{align}
  b(\uu_h,\Pi_F\vv) = b(\uu_h,\vv) \quad\forall \uu_h\in U_h, \vv\in V.
\end{align}
\end{subequations}
For the particular choice of spaces that will be used in our numerical examples we verify the existence of such a Fortin
operator in Section~\ref{sec:fullydiscrete}.

The trial-to-test operator is replaced by its discrete version $\Theta_h\colon U_h\to V_h$ given by
\begin{align*}
  \ip{\Theta_h \uu_h}{\vv_h}_V = b(\uu_h,\vv_h) \quad\forall \vv_h\in V.
\end{align*}

The fact that Fortin operators imply well-posedness of mixed finite element schemes is
well-known~\cite{BoffiBrezziFortin}.
For DPG methods, which can be  rewritten as mixed formulations, such a result is explicitly stated
in~\cite[Theorem~2.1]{practicalDPG}.
It follows
\begin{theorem}
  Let $U_h\subset U$, $V_h\subset V$ such that a Fortin operator exists, i.e.,~\eqref{eq:fortin} is satisfied. Then, the
  problem
  \begin{align}\label{eq:dpg:fullydiscrete}
    \uu_h \in U_h : \quad
    b(\uu_h,\Theta_h\ww_h) = F(\Theta_h\ww_h) \quad\forall \ww_h \in U_h
  \end{align}
  admits a unique solution.
  
  Let $\uu\in U$ denote the solution of~\eqref{eq:ultraweak} and let $\uu_h\in U_h$ denote the solution
  of~\eqref{eq:dpg:fullydiscrete}, then
  \begin{align}\label{eq:quasiopt:fullydiscrete}
    \norm{\uu-\uu_h}U \leq C_\mathrm{qopt} C_F \min_{\ww_h\in U_h} \norm{\uu-\ww_h}U.
  \end{align}
\end{theorem}

\subsection{DPG--Least-squares coupling method}\label{sec:dpglsqmethod}
Another possibility is to combine a least-squares formulation and the ultraweak formulation:
\begin{align}
  \uu^\star = \arg\min_{\uu\in U} \left(
  \norm{\CC\uu}{\HdivDivset\TT'}^2 + \norm{A:\MM-f}{}^2 \right),
\end{align}
where $\CC\colon U\to \HdivDivset\TT'$ is the operator corresponding to the bilinear form
\begin{align}\label{eq:defC}
  c(\uu,\QQ) = \ip{u}{-\dDiv\QQ}_\TT +  \ip{\MM}{\QQ}+\dual{\widehat\uu}\QQ_\cS.
\end{align}
The Euler--Lagrange equations read: $\uu\in U:$
\begin{align}\label{eq:dpglsq:eulerlagrange}
  c(\uu,\Theta^{\dDiv}\ww) + \ip{A:\MM}{A:\ZZ} = \ip{f}{A:\ZZ} \quad\forall \ww = (w,\ZZ,\widehat\ww)\in U,
\end{align}
where $\Theta^{\dDiv}\colon U \to \HdivDivset\TT$ is the trial-to-test operator defined via
\begin{align*}
  \ip{\Theta^{\dDiv}\ww}{\QQ}_{\HdivDivset\TT} = c(\ww,\QQ) \quad\forall\QQ\in \HdivDivset\TT.
\end{align*}
We employed a similar idea in~\cite{DPGBEM} for the coupling of least-squares boundary elements methods and the DPG method.

In the following we use the notation $Q = \HdivDivset\TT$.
We note that
\begin{align*}
  \norm{\CC\uu}{Q'} = \sup_{0\neq \QQ\in Q} \frac{c(\uu,\QQ)}{\norm{\QQ}{Q}} 
  = c(\uu,\Theta^{\dDiv}\uu)^{1/2}
\end{align*}
by some standard arguments.

A proof of the next result is found in Section~\ref{sec:proofs} (and follows from the observation on the equivalence to
the DPG method given in Section~\ref{sec:equivalence} below).
\begin{theorem}\label{thm:dpglsq}
  Let $f\in L^2(\Omega)$. Problem~\eqref{eq:dpglsq:eulerlagrange} admits a unique solution $\uu^\star\in U$ which satisfies
  \begin{align*}
    \norm{\uu^\star}U \leq C \norm{f}{}.
  \end{align*}

  Let $U_h\subset U$ denote a finite dimensional subspace. Then, the problem
  \begin{align}\label{eq:lsq:discrete}
    \uu_h \in U_h : \quad
    c(\uu_h,\Theta^{\dDiv}\ww_h) + \ip{A:\MM_h}{A:\ZZ_h} = \ip{f}{A:\ZZ_h} \quad\forall \ww_h \in U_h
  \end{align}
  admits a unique solution (here, $\uu_h=(u_h,\MM_h,\widehat\uu_h)$ and $\ww_h = (w_h,\ZZ_h,\widehat\ww_h)$).

  Moreover, let $\uu\in U$ and $\uu_h\in U$ denote the solutions of~\eqref{eq:dpglsq:eulerlagrange} and~\eqref{eq:lsq:discrete}
  respectively. Then,
  \begin{align*}
    \norm{\uu-\uu_h}U \leq C_\mathrm{qopt} \min_{\ww_h\in U_h} \norm{\uu-\ww_h}U.
  \end{align*}
\end{theorem}

As before we consider finite dimensional spaces $U_h\subset U$ and $Q_h\subset Q$ and replace the optimal test-functions
by discretized ones, i.e., consider the discrete operator $\Theta_h^{\dDiv}\colon U_h \to Q_h$,
\begin{align*}
  \ip{\Theta_h^{\dDiv}\uu_h}{\QQ_h}_Q = c(\uu_h,\QQ_h) \quad\forall \QQ_h\in Q_h.
\end{align*}

A Fortin operator for this problem is an operator $\Pi_F\colon Q\to Q_h$ such that
\begin{align}\label{eq:fortin:lsq}
  c(\uu_h,\QQ) = c(\uu_h,\Pi_F\QQ), \quad \norm{\Pi_F\QQ}{Q}\leq C_F \norm{\QQ}Q \quad\forall
  \uu_h\in U_h, \, \QQ\in Q.
\end{align}

The proof of the following result is postponed to Section~\ref{sec:proofs}.
\begin{theorem}\label{thm:dpglsq:discrete}
  Let $U_h\subset U$, $Q_h\subset Q$ such that a Fortin operator exists, i.e.,~\eqref{eq:fortin:lsq} is satisfied. Then, the
  problem
  \begin{align}\label{eq:lsq:fullydiscrete}
    \uu_h \in U_h : \quad
    c(\uu_h,\Theta_h^{\dDiv}\ww_h) + \ip{A:\MM_h}{A:\ZZ_h} = \ip{f}{A:\ZZ_h} \quad\forall \ww_h \in U_h
  \end{align}
  admits a unique solution.
  
  Let $\uu\in U$ denote the solution of~\eqref{eq:dpglsq:eulerlagrange} and let $\uu_h\in U_h$ denote the solution
  of~\eqref{eq:lsq:fullydiscrete}, then
  \begin{align}\label{eq:quasiopt:lsq:fullydiscrete}
    \norm{\uu-\uu_h}U \leq C_\mathrm{qopt} \max\{C_F,1\} \min_{\ww_h\in U_h} \norm{\uu-\ww_h}U.
  \end{align}
\end{theorem}

\subsection{Equivalence of DPG and DPG--Least-squares method}\label{sec:equivalence}
We stress that the DPG--Least-squares coupling is only a special representation of the DPG method. To see this, consider
the trial-to-test operator $\Theta : U\to V'$:
For given $\ww = (w,\ZZ,\widehat\ww)\in U$ we compute $\Theta\ww = \vv = (v,\QQ)\in V$ by
\begin{align*}
  \ip{(v,\QQ)}{(\delta v,\delta\QQ)}_V = b(\ww,(\delta v,\delta\QQ)),
\end{align*}
and for $\delta\QQ=0$ we obtain that
\begin{align*}
  \ip{v}{\delta v} = b(\ww,(\delta v,0)) = \ip{A:\ZZ}{\delta v} \quad\forall \delta v\in L^2(\Omega)
\end{align*}
and therefore $v = A:\ZZ$.
On the other hand, if we test with $\delta v=0$, then,
\begin{align*}
  \ip{\QQ}{\delta\QQ}_Q = b(\ww,(0,\delta\QQ)) = c(\ww,\delta\QQ) \quad\forall \delta\QQ\in Q,
\end{align*}
which means that $\QQ = \Theta^{\dDiv}\ww$. 
These observations yield that for all $\uu,\ww\in U$
\begin{align*}
  b(\uu,\Theta\ww) &= b(\uu,(v,\QQ)) = b(\uu,(0,\QQ)) + \ip{A:\MM}{A:\ZZ}
  = c(\uu,\QQ) + \ip{A:\MM}{A:\ZZ} \\
  &= c(\uu,\Theta^{\dDiv}\ww) + \ip{A:\MM}{A:\ZZ}.
\end{align*}

The only difference between the methods is when it comes to the fully discrete schemes:
For the DPG method (Section~\ref{sec:dpgmethod}) we have two components when computing (discrete) optimal test functions
$\Theta_h\uu_h = (\vv_h,\QQ_h)\in N_h\times Q_h$, whereas for the
DPG--Least-squares scheme (Section~\ref{sec:dpglsqmethod}) we only have one $\QQ_h\in Q_h$.
Nevertheless, we can recover the DPG--Least-squares scheme from the DPG method with the same argumentation as above.
Let $M_h\subset \LL_\mathrm{sym}^2(\Omega)$ denote some finite-dimensional subspace to approximate the matrix-valued
solution component and consider
\begin{align*}
  N_h := \set{A:\MM_h}{\MM_h\in M_h}\subset L^2(\Omega).
\end{align*}
Then, the same calculations as above show that the two methods are equivalent (see also Section~\ref{sec:proofs:dpglsq} for
more details).
However, observe that in practice it is hard to determine a basis for the space $N_h$.

\section{Analysis of the ultraweak formulations}\label{sec:proofs}
In this section we present proofs for the main results of Section~\ref{sec:varforms}.
Here, we follow the concept of ``breaking spaces'' introduced in~\cite{breakSpace} for the proof of Theorem~\ref{thm:dpg}. The
proof of Theorem~\ref{thm:dpglsq} is then a simple corollary.

\subsection{Global adjoint problem}

\begin{lemma}\label{lem:adjoint}
  Let $g\in L^2(\Omega)$, $\GG\in \LL_\mathrm{sym}^2(\Omega)$. Then, the problem
  \begin{align*}
    -\dDiv\QQ &= g, \\
    Av + \QQ &= \GG, \\
    \normal\cdot\QQ\normal|_\Gamma &= 0.
  \end{align*}
  admits a unique solution $(v,\QQ)\in L^2(\Omega) \times \HdivDivset\Omega$.

  Moreover,
  \begin{align}
    \norm{v}{} + \norm{\QQ}{\HdivDivset\Omega} \lesssim \norm{g}{} + \norm{\GG}{}.
  \end{align}
\end{lemma}
\begin{proof}
  We define the variational problem
  \begin{align}\label{eq:adjoint:proof:1}
    v \in L^2(\Omega) : \quad \ip{v}{A:D^2 z} = \ip{g}z + \ip{\GG}{D^2 z} \quad\forall z\in X= H^2(\Omega)\cap H_0^1(\Omega).
  \end{align}
  This problem admits a unique solution since $a(z,v) := \ip{A:D^2 z}{v}$ is bounded on $X\times L^2(\Omega)$ 
  and satisfies the $\inf$--$\sup$ conditions (Lemma~\ref{lem:blfa}).
  Now, let $v\in L^2(\Omega)$ denote the solution of~\eqref{eq:adjoint:proof:1}. Then, we have that
  \begin{align*}
    \norm{v}{} \lesssim \norm{g}{} + \norm{\GG}{}.
  \end{align*}
  We define $\QQ \in \LL_\mathrm{sym}^2(\Omega)$ via the relation $Av+\QQ = \GG$. It remains to verify that $\dDiv\QQ
  = -g$ and $\normal\cdot\QQ\normal|_\Gamma = 0$.
  Taking $\QQ = \GG-Av$ in~\eqref{eq:adjoint:proof:1} gives us
  \begin{align*}
    -\ip{\QQ}{D^2 z} = \ip{g}z \quad\forall z\in X.
  \end{align*}
  To see that $\dDiv\QQ = -g \in L^2(\Omega)$ take $z\in \DD(\Omega)$ in the last identity. Then, 
  \begin{align*}
    \dDiv\QQ(z) = \ip{\QQ}{D^2z} = - \ip{g}z
  \end{align*}
  shows that $\dDiv\QQ = -g$.

  Finally, recall the definitions of $\normal\cdot\QQ\normal$ and
  $\trdDiv_\Omega\QQ$. Using $\QQ=\GG-Av$ and~\eqref{eq:adjoint:proof:1} again we get that
  \begin{align*}
    -\dual{\normal\cdot\QQ\normal}{\partial_\normal z}_\Gamma = \dual{\trdDiv\QQ}{z}_\Gamma = 
    \ip{\dDiv\QQ}z - \ip{\QQ}{D^2 z} = \ip{-g}z - \ip{\GG-Av}{D^2z} = 0
  \end{align*}
  for all $z\in X$ which shows that $\normal\cdot\QQ\normal|_\Gamma = 0$.

  The solution to the mixed problem is also unique: Suppose $g = 0$ and $\GG = 0$ and that the pair $(v,\QQ)$ is a
  solution to the mixed formulation. Testing the first equation with $z\in H^2(\Omega)\cap H_0^1(\Omega)$, integration
  by parts, replacing $\QQ$ with $-Av$ and the boundary condition $\normal\cdot\QQ\normal=0$ show that $v\in
  L^2(\Omega)$ satisfies~\eqref{eq:adjoint:proof:1} with right-hand side equal to zero. Consequently, $v=0$ and $\QQ=-Av
  = 0$.
\end{proof}

\subsection{Trace spaces}
A thorough analysis of the trace spaces used in the present work is found in~\cite{KLove1,ReissnerMindlinDPG}.
We only need the following lemma where its proof is a small modification of~\cite[Proposition~3.8]{KLove1} but follows the very
same steps, see also~\cite[Proposition~11]{ReissnerMindlinDPG}.
Therefore, we omit the proof.
\begin{lemma}\label{lem:traces}
  Let $\QQ\in\HdivDivset\TT$. Then,
  \begin{align*}
    \QQ \in \HdivDivset\Omega \text{ with } \normal\cdot\QQ\normal|_\Gamma = 0 
    \Longleftrightarrow
    \dual{\widehat\uu}{\QQ}_\cS = 0 \quad\forall \widehat\uu \in \widehat U.
  \end{align*}
\end{lemma}

\subsection{Putting together}
To actually show Theorem~\ref{thm:dpg} we verify the assumptions of~\cite[Theorem~3.3]{breakSpace}. 
We give the results in the notation from the present work.

Let $U_0 = L^2(\Omega)\times\LL_\mathrm{sym}^2(\Omega)$. Clearly, $U = U_0\times \widehat U$.
Let $V_0 = L^2(\Omega)\times \HdivDivNset\Omega$. 
Define the bilinear form $b_0\colon U_0\times V_0 \to \R$ by
\begin{align*}
  b_0( (u,\MM), (v,\QQ) ) := \ip{u}{-\dDiv\QQ} + \ip{\MM}{Av+\QQ}.
\end{align*}

\begin{proposition}\label{prop:infsup}
  It holds that
  \begin{align*}
    \norm{u}{} + \norm{\MM}{}\lesssim \sup_{0\neq (v,\QQ)\in V_0} 
    \frac{b_0( (u,\MM),(v,\QQ))}{\big(\norm{v}{}^2+\norm{\QQ}{\HdivDivset\Omega}^2\big)^{1/2}} 
    \quad\forall (u,\MM)\in U_0,
  \end{align*}
  and
  \begin{align*}
    \set{(v,\QQ)\in V_0}{b_0( (u,\MM),(v,\QQ))=0 \quad \forall (u,\MM)\in U_0} = \{0\}.
  \end{align*}
\end{proposition}
\begin{proof}
  The $\inf$--$\sup$ condition follows from Lemma~\ref{lem:adjoint}: Let $(u,\MM)\in U_0$ be given and choose $g=u$,
  $\GG = \MM$ in Lemma~\ref{lem:adjoint} and let $(v,\QQ)\in V_0$ denote the solution to the system from
  Lemma~\ref{lem:adjoint}. Then,
  \begin{align*}
    \norm{u}{}^2 + \norm{\MM}{}^2 &= b_0( (u,\MM),(v,\QQ)) = 
    \frac{b_0((u,\MM),(v,\QQ))}{\big(\norm{v}{}^2+\norm{\QQ}{\HdivDivset\Omega}^2\big)^{1/2}}
    \big(\norm{v}{}^2+\norm{\QQ}{\HdivDivset\Omega}^2\big)^{1/2} \\
    &\lesssim \frac{b_0((u,\MM),(v,\QQ))}{\big(\norm{v}{}^2+\norm{\QQ}{\HdivDivset\Omega}^2\big)^{1/2}}
    \big(\norm{u}{}^2+\norm{\MM}{}^2\big)^{1/2}.
  \end{align*}
  Dividing by $\big(\norm{u}{}^2+\norm{\MM}{}^2\big)^{1/2}$ and taking the supremum over $V_0$ finishes the proof of the
  $\inf$--$\sup$ condition. 
  
  To see the last assertion suppose that $(v,\QQ)\in V_0$ such that $b_0( (\cdot,\cdot),(v,\QQ)) = 0$, i.e.,
  \begin{align*}
    \ip{u}{-\dDiv\QQ} + \ip{\MM}{Av+\QQ} = 0 \quad\forall (u,\MM)\in U_0 = L^2(\Omega)\times \LL_\mathrm{sym}^2(\Omega).
  \end{align*}
  Take $u = -\dDiv\QQ$ and $\MM = Av+\QQ$. Then,
  \begin{align*}
    \norm{\dDiv\QQ}{}^2 + \norm{Av+\QQ}{}^2 = 0
  \end{align*}
  or equivalently $\div\Div\QQ = 0$ and $Av+\QQ=0$. By Lemma~\ref{lem:adjoint} this homogeneous problem has a unique
  solution equal to $0$ which concludes the proof.
\end{proof}

Define the bilinear form $\widehat b\colon \widehat U \times V \to \R$ by
\begin{align*}
  \widehat b(\widehat\uu,(v,\QQ)) = \dual{\widehat\uu}{\QQ}_{\cS}.
\end{align*}

We note that Lemma~\ref{lem:traces} can also be stated as
\begin{proposition}\label{prop:traces2}
  It holds that
  \begin{align*}
    V_0 = \set{v\in V}{\widehat b(\widehat\uu,\vv) = 0 \quad\forall 
    \widehat\uu \in \widehat U}.
  \end{align*}
\end{proposition}

\textbf{Proof of Theorem~\ref{thm:dpg}.}
We note that Proposition~\ref{prop:traces}, Proposition~\ref{prop:infsup}, and Proposition~\ref{prop:traces2} verify the
assumptions of~\cite[Theorem~3.3]{breakSpace}.
In particular, this implies that problem~\eqref{eq:ultraweak} is well-posed.
By the DPG theory this yields that the semi-discrete problem~\eqref{eq:dpg:discrete} admits a unique solution and the
quasi-optimality stated in Theorem~\ref{thm:dpg}. \qed

\subsection{Analysis of the DPG-Least-squares scheme}\label{sec:proofs:dpglsq}
Theorem~\ref{thm:dpglsq} follows from Theorem~\ref{thm:dpg} and the observation on the equivalence of the two schemes
from Section~\ref{sec:equivalence}.

The results on the fully-discrete scheme from Theorem~\ref{thm:dpglsq:discrete} can be seen as follows.
By assumption there exists an operator $\Pi^{\dDiv} : Q\to Q_h$ with $c(\uu_h,\QQ) = c(\uu_h,\Pi^{\dDiv}\QQ)$ and
$\norm{\Pi^{\dDiv}\QQ}{Q}\leq C_F \norm{\QQ}{Q}$.
Consider the space
\begin{align*}
  V_h = \set{A:\MM}{\MM \in \LL_\mathrm{sym}^2(\Omega)} \times Q_h
  = L^2(\Omega) \times Q_h.
\end{align*}
It is straightforward to show that $\Pi_F\colon V \to V_h$ given by 
\begin{align*}
  \Pi_F(v,\QQ) = (v,\Pi^{\dDiv}\QQ) \quad\text{satisfies}\quad
  \norm{\Pi_F\vv}{V} \leq \max\{C_F,1\} \norm{\vv}V
\end{align*}
and
\begin{align*}
  b(\uu_h,\Pi_F\vv) = b(\uu_h,\vv) \quad\text{for all } \uu_h\in U_h, \vv\in V.
\end{align*}
The same argumentation as in Section~\ref{sec:equivalence} then shows that
\begin{align*}
  b(\uu_h,\Theta_h\ww_h) = c(\uu_h,\Theta_h^{\dDiv}\ww_h) + \ip{A:\MM_h}{A:\ZZ_h}
  \quad\forall \uu_h,\ww_h \in U_h
\end{align*}
and
\begin{align*}
  F(\Theta_h\ww_h) = \ip{f}{A:\ZZ_h} \quad\forall \ww_h\in U_h.
\end{align*}
Therefore, 
Theorem~\ref{thm:dpg:discrete} implies the assertions of Theorem~\ref{thm:dpglsq:discrete}. \qed

\section{A posteriori estimators}\label{sec:aposteriori}
In the following we define a posteriori estimates for the two numerical schemes introduced in this work and state their
efficiency and reliability.
\begin{theorem}\label{thm:est:dpg}
  Suppose that $U_h\subset U$, $V_h\subset V$ and that there exists a Fortin operator~\eqref{eq:fortin} with $\Pi_F\vv =
  (\Pi^{L^2}v,\Pi^{\dDiv}\QQ)\in V_h$ for $\vv=(v,\QQ)\in V$.
  Let $\uu\in U$ and $\uu_h\in U_h$ denote the solution of~\eqref{eq:ultraweak} and~\eqref{eq:dpg:fullydiscrete},
  respectively. Then,
  \begin{align}
    \norm{\uu-\uu_h}U^2 \simeq \etaDPG^2:= \norm{F-B\uu_h}{V_h'}^2 + \norm{A:\MM_h-f}{}^2.
  \end{align}
\end{theorem}
\begin{proof}
  By~\cite[Theorem~2.1]{DPGaposteriori} we have that
  \begin{align*}
    \norm{\uu-\uu_h}U^2 \simeq \norm{F-B\uu_h}{V_h'}^2 + \osc(F)^2,
  \end{align*}
  where the oscillation term is defined as
  \begin{align*}
    \osc(F) := \sup_{0\neq \vv\in V} \frac{F(\vv-\Pi_F\vv)}{\norm{\vv}V}.
  \end{align*}
  It only remains to show that $\osc(F)\lesssim \norm{A:\MM_h-f}{}\lesssim \norm{\uu-\uu_h}U$.
  Recall that $F(\vv) = \ip{f}v$ for $\vv=(v,\QQ)\in V$. From the Fortin property~\eqref{eq:fortin} we deduce that
  $\ip{A:\MM_h}{v-\Pi^{L^2}v} = 0$ for all $\uu_h = (0,\MM_h,0)\in U_h$, $\vv = (v,0)\in V$. This yields that
  \begin{align*}
    \sup_{0\neq \vv\in V} \frac{F(\vv-\Pi_F\vv)}{\norm{\vv}V} &= \sup_{0\neq v\in L^2(\Omega)}
    \frac{\ip{f}{v-\Pi^{L^2}v}}{\norm{v}{}}
    \\& = \sup_{0\neq v\in L^2(\Omega)} \frac{\ip{f-A:\QQ_h}{v-\Pi^{L^2}v}}{\norm{v}{}}
    \lesssim \norm{f-A:\QQ_h}{}
  \end{align*}
  for any $\QQ_h\in\LL^2_\mathrm{sym}(\Omega)$. Choosing $\QQ_h=\MM_h$ and using that $f = A:D^2 u = A:\MM$ we get that
  \begin{align*}
    \osc(F) \lesssim \norm{f-A:\MM_h}{} = \norm{A:(\MM-\MM_h)}{} \lesssim \norm{\MM-\MM_h}{} \leq \norm{\uu-\uu_h}U
  \end{align*}
  which finishes the proof.
\end{proof}

\begin{corollary}\label{cor:est:dpg}
  With the same notation and assumptions as in Theorem~\ref{thm:est:dpg} it holds that
  \begin{align*}
    \norm{\uu-\uu_h}U^2 \simeq \etaLSQ^2 := \norm{\CC\uu_h}{Q_h'}^2 
    + \norm{A:\MM_h-f}{}^2.
  \end{align*}
\end{corollary}
\begin{proof}
  Starting from Theorem~\ref{thm:est:dpg} we have that (using the notation $V_h=W_h\times Q_h$)
  \begin{align*}
    \norm{F-B\uu_h}{V_h'} &= \sup_{0\neq \vv_h=(v_h,\QQ_h)\in V_h} \frac{\ip{f}{v_h}-b(\uu_h,\vv_h)}{\norm{\vv_h}{V}}
    \\
    &= \sup_{0\neq \vv_h=(v_h,\QQ_h)\in V_h} \frac{\ip{f-A:\MM_h}{v_h}-c(\uu_h,\QQ_h)}{\norm{\vv_h}{V}} \\
    &\leq \sup_{0\neq \QQ_h\in Q_h} \frac{c(\uu_h,\QQ_h)}{\norm{\QQ_h}{\HdivDivset\TT}} 
    + \sup_{0\neq v_h\in W_h} \frac{\ip{f-A:\MM_h}{v_h}}{\norm{v_h}{}} 
    \\
    &\leq \norm{\CC\uu_h}{Q_h'} +
    \norm{f-A:\MM_h}{}.
  \end{align*}
  It only remains to show that $\norm{\CC\uu_h}{Q_h'}\lesssim \norm{\uu-\uu_h}U$ since in the proof of
  Theorem~\ref{thm:est:dpg} it was already shown that $\norm{f-A:\MM_h}{} \lesssim \norm{\uu-\uu_h}U$.
  Note that the exact solution can be written as $\uu = (u,\MM,\widehat\uu) =  (u,D^2u,\trHtwo_\TT u)$ and that
  $c(\uu,\QQ) = 0$ for all $\QQ\in \HdivDivset\TT$, hence, $\uu\in \ker(\CC)$. Finally, this together with boundedness of
  the operator $\CC$
  implies that
  \begin{align*}
    \norm{\CC\uu_h}{Q_h'} = \norm{\CC(\uu-\uu_h)}{Q_h'} \leq \norm{\CC(\uu-\uu_h)}{\HdivDivset\TT'} \lesssim
    \norm{\uu-\uu_h}U
  \end{align*}
  which finishes the proof.
\end{proof}

For the DPG-Least-squares scheme the same estimator as given in Corollary~\ref{cor:est:dpg} can be used.
The proof of the following result is similar to the one of Theorem~\ref{thm:est:dpg} and Corollary~\ref{cor:est:dpg},
and is therefore left to the reader.
\begin{theorem}\label{thm:est:lsq}
  Suppose that $U_h\subset U$, $Q_h\subset Q$ and that there exists a Fortin operator~\eqref{eq:fortin:lsq}.
  Let $\uu\in U$ and $\uu_h\in U_h$ denote the solution of~\eqref{eq:dpglsq:eulerlagrange} and~\eqref{eq:lsq:fullydiscrete},
  respectively. Then,
  \begin{align}
    \norm{\uu-\uu_h}U^2 \simeq \etaLSQ^2:= \norm{\CC\uu_h}{Q_h'}^2 + \norm{A:\MM_h-f}{}^2.
  \end{align}
\end{theorem}

\section{Fortin operators}\label{sec:fullydiscrete}
We restrict the presentation of Fortin operators to the lowest-order case (for the trial space) and $d=2$ which allows
us to use results established in~\cite{KLove2}.

\subsection{Discretization}
Let $\TT$ denote a shape-regular triangulation of the domain $\Omega\subset \R^2$.
With $\PP^p(T)$ we denote the space of polynomials on $T\in\TT$ with degree less or equal to $p\in\N_0$.
We consider the space
\begin{align*}
  \PP^p(\TT) = \set{v\in L^2(\Omega)}{v|_T\in \PP^p(T) \,\, \forall T\in\TT}.
\end{align*}

For a triangle $T\in\TT$ we denote with $\EE_T$ the set of its edges, $\EE := \bigcup_{T\in\TT} \EE_T$.
Similar to the definition of $\PP^p(T)$ we denote with $\PP^p(E)$ the set of polynomials on $E\in\EE$ with degree less
or equal to $p\in\N_0$ and 
\begin{align*}
  \PP^p(\EE_T) := \set{v\in L^2(\partial T)}{v|_E \in \PP^p(E) \,\,\forall E\in\EE_T}.
\end{align*}

For $T\in\TT$ we define the local space
\begin{align*}
  U_T := \set{v\in H^2(T)}{\Delta^2 v + v = 0, \, v|_{\partial T} \in \PP^3(\EE_T), \, 
  \normal_T\cdot\nabla v|_{\partial T} \in \PP^1(\EE_T)}
\end{align*}
and the local trace space
\begin{align*}
  \widehat U_T := \trHtwo_T(U_T).
\end{align*}
We note that $\dim(\widehat U_T) = \dim(U_T) = 9$ and that $\widehat U_T$ is also the trace space of the
rHCT element (cf.~\cite{KLove1}). The local degrees of freedom of this element are associated to the nodal values of the
trace and the nodal values of the trace of the gradient, $\{(v(z),\nabla v(z))\,:\,z \text{ is vertex of } T\}$.
The global approximation space is then given by
\begin{align*}
  \widehat U_h = \set{\widehat\vv\in\widehat U}{\widehat\vv|_{\partial T} \in \widehat U_T \quad\forall T\in\TT}.
\end{align*}
We investigate the boundary condition. To that end let $\VV_\Gamma$ denote the boundary vertices of the triangulation
$\TT$. We decompose $\VV_\Gamma = \VV_c \cup \VV_0$ where $\VV_c$ denotes the set of all corner vertices, i.e., all
vertices where the (interior) angle between adjacent edges is strictly less than $\pi$.
Consequently, $\VV_0$ is then the set of boundary vertices where the angle between adjacent edges equals $\pi$ or in
other words the tangential vectors of the edges are equal.
Let $\widehat\vv\in \widehat U_h$. Note that we can identify $\widehat\vv|_\Gamma$ with the set $\{(v(z),\nabla
v(z))\,:\,z\in\VV_\Gamma\}$.
Recall that $v|_\Gamma$ is a polynomial of degree less or equal than three and is determined by its nodal values and the
nodal values of its gradient. In particular, $v|_\Gamma = 0$ is equivalent to
\begin{align*}
  v(z) &= 0 \quad\forall z\in \VV_\Gamma, \\
  \tangential(z)\cdot\nabla v(z) &=0 \quad\forall z\in \VV_0, \\
  \nabla v(z) &=0 \quad\forall z\in \VV_c.
\end{align*}
Here, $\tangential(z)$ denotes the tangential vector in $z\in\Gamma$ which is well-defined for $z\in
\VV_0$. 

\begin{lemma}\label{lem:approx}
  For $u\in H^3(\Omega)$ we have that
  \begin{align*}
    \min_{\widehat\vv_h\in \widehat U_h}\norm{\trHtwo_\TT u - \widehat\vv_h}{3/2,1/2,\cS} \leq C h \norm{u}{H^3(\Omega)}.
  \end{align*}
\end{lemma}
\begin{proof}
  The trace theorem implies that
  \begin{align*}
    \norm{\trHtwo_\TT u-\widehat\vv_h}{3/2,1/2,\cS} \leq \norm{u-v_h}{H^2(\Omega)}
  \end{align*}
  for all $\widehat\vv_h$ and $v_h$ with $\trHtwo_\TT v_h = \widehat\vv_h$.
  Interpolation error estimates for rHCT elements (see~\cite{CiarletHCT}) give
  \begin{align*}
    \norm{u-v_h}{H^2(\Omega)} \lesssim h\norm{u}{H^3(\Omega)}
  \end{align*}
  which concludes the proof.
\end{proof}

Using the discrete trial space
\begin{align*}
U_h := \PP^0(\TT)\times (\PP^0(\TT)^{2\times 2}\cap\LL^2_\mathrm{sym}(\Omega)) \times \widehat U_h
\end{align*}
we conclude together with standard approximation results ($\norm{(1-\Pi^0)w}{}\lesssim h\norm{\nabla w}{}$) the
following result:
\begin{corollary}
  Let $u\in H^3(\Omega)\cap X$ and $\uu = (u,D^2 u,\trHtwo_\TT u)\in U$. Then,
  \begin{align*}
    \min_{\ww_h\in U_h} \norm{\uu-\ww_h}{U} \lesssim h \norm{u}{H^3(\Omega)}.
  \end{align*}
\end{corollary}

\begin{remark}\label{rem:poly}
  It is also possible to define the approximation spaces for polygonal shaped elements.
  Consider a polygonal shaped element $K$ with the set of its edges $\EE_K$ and the space
  \begin{align*}
    U_K := \set{v\in H^2(K)}{\Delta^2 v + v = 0, \, v|_{\partial K} \in \PP^3(\EE_K), \, 
    \normal_K\cdot\nabla v|_{\partial K} \in \PP^1(\EE_K)}.
  \end{align*}
  We then follow the same ideas as presented above by defining the local trace space $\widehat U_K := \trHtwo_K(U_K)$. 
  The local degrees of freedom are associated to the nodal values and nodal values of the gradient which gives
  $\dim(\widehat U_K) = \dim(U_K) = 3\#(\text{vertices of }K)$.
  To obtain approximation results similar to the ones in Lemma~\ref{lem:approx} one needs some restriction on the shape of the element
  $K$. We refer to~\cite{PolyDPG} where all details have been worked out for Poisson's equation (trace spaces of
  $H^1(\Omega)$ functions).
\end{remark}

\begin{remark}\label{rem:higherorder}
  It is also possible to define spaces with higher-order approximation properties.
  Consider
  \begin{align*}
    U_T^p := \begin{cases}
        \set{v\in H^2(T)}{\Delta^2 v + v = 0, \, v|_{\partial T} \in \PP^3(\EE_T), \, 
        \normal_T\cdot\nabla v|_{\partial T} \in \PP^1(\EE_T)} & p=1, \\
        \set{v\in H^2(T)}{\Delta^2 v + v = 0, \, v|_{\partial T} \in \PP^3(\EE_T), \, 
        \normal_T\cdot\nabla v|_{\partial T} \in \PP^2(\EE_T)} & p=2, \\
        \set{v\in H^2(T)}{\Delta^2 v + v = 0, \, v|_{\partial T} \in \PP^{p+1}(\EE_T), \, 
        \normal_T\cdot\nabla v|_{\partial T} \in \PP^{p}(\EE_T)} & p\geq 3,
    \end{cases}
  \end{align*}
  and $\widehat U_T^p = \trHtwo_T(U_T^p)$. Then, for $p=1$ we recover the space $\widehat U_T$. We note that elements of 
  $\widehat U_T^2$ are the traces of HCT elements. The approximation order is $h^2$.
  Without further details we stress that the approximation order of $\widehat U_T^p$ (and the corresponding global
  space) is $h^p$.
\end{remark}

\subsection{Fortin operator}\label{sec:fortin}
We start by citing a result from~\cite{KLove2}. Recall that $Q=\HdivDivset\TT$. We use the discrete space
$Q_h := \PP^4(\TT)^{2\times 2}\cap \LL^2_\mathrm{sym}(\Omega) \subset Q$.
Let $\Pi^p\colon L^2(\Omega)\to \PP^p(\TT)$ denote the $L^2(\Omega)$-projection.

\begin{lemma}[{\cite[Lemma~16]{KLove2}}]\label{lem:fortin}
  There exists $\Pi^{\dDiv}\colon Q\to Q_h$ such that
  \begin{alignat*}{2}
    \dual{\widehat u}{\Pi^{\dDiv}\QQ}_\cS &= \dual{\widehat u}{\QQ}_\cS &\quad& \forall \widehat\uu\in \widehat U_h, \\
    \ip{\MM}{\Pi^{\dDiv}\QQ} &= \ip{\MM}{\QQ} &\quad& \forall \MM\in \PP^0(\TT)^{2\times
    2}\cap\LL^2_\mathrm{sym}(\Omega) \\
    \ip{u}{\div\Div\Pi^{\dDiv}\QQ}_\TT &= \ip{u}{\div\Div\QQ}_\TT 
    &\quad&\forall u\in \PP^2(\TT)
  \end{alignat*}
  for any $\QQ\in Q$. 

  Moreover, 
  \begin{align*}
    \norm{\Pi^{\dDiv}\QQ}{\HdivDivset\TT} \lesssim \norm{\QQ}{\HdivDivset\TT}
    \quad\forall \QQ\in Q.
  \end{align*}
\end{lemma}
It immediately follows:
\begin{corollary}
  The operator $\Pi^{\dDiv}\colon Q\to Q_h$ (Lemma~\ref{lem:fortin}) is a Fortin operator for the problem from Section~\ref{sec:dpglsqmethod},
  i.e.,~\eqref{eq:fortin:lsq} is satisfied with $\Pi_F=\Pi^{\dDiv}$ and $Q_h$, $U_h$ as defined in this section.
\end{corollary}

Define the discrete test space
\begin{align*}
  V_{hp} := \PP^p(\TT) \times Q_h.
\end{align*}

\begin{theorem}\label{thm:fortin}
  Suppose that $A\in \PP^p(\TT)^{d\times d}\cap\LL^2_\mathrm{sym}(\Omega)$. 
  Then, $\Pi_F = (\Pi^p,\Pi^{\dDiv})\colon V\to V_{hp}$ is a Fortin operator for the problem from
  Section~\ref{sec:dpgmethod}, i.e.,~\eqref{eq:fortin} is satisfied.
\end{theorem}
\begin{proof}
  Since $\MM_h \in \PP^0(\TT)^{d\times d}$ this shows that $A:\MM_h \in \PP^p(\TT)$ and therefore
  \begin{align*}
    \ip{A:\MM_h}{\Pi^pv} = \ip{A:\MM_h}v \quad
    \forall \MM_h\in \PP^0(\TT)\cap \LL^2_\mathrm{sym}(\Omega), \, 
    v\in L^2(\Omega).
  \end{align*}
  Moreover, $\norm{\Pi^p v}{}\leq \norm{v}{}$.
  Together with Lemma~\ref{lem:fortin} we conclude that
  \begin{align*}
    b(\uu_h,\Pi_F \vv) = b(\uu_h,\vv) \quad\forall \vv\in V
  \end{align*}
  and $\norm{\Pi_F \vv}{V} \lesssim \norm{\vv}V$ which finishes the proof.
\end{proof}

\section{Numerical Studies}\label{sec:examples}
In this section we present several numerical experiments showing the performance of our proposed methods.
Throughout we consider the computational domain $\Omega = (-1,1)^2$ and the initial triangulation is shown in
Figure~\ref{fig:mesh}.
Throughout we use the test space $V_{h0}$ and $Q_h$ defined in Section~\ref{sec:fortin} for the methods from
Section~\ref{sec:dpgmethod} and Section~\ref{sec:dpglsqmethod}, respectively.
For both methods we use the trial space $U_h$ defined above for the examples from
Section~\ref{sec:ex1}---\ref{sec:ex2}.
In Section~\ref{sec:ex4} we consider the augmented trial space
\begin{align*}
  U_h^+ = \PP^1(\TT) \times \big(\PP^0(\TT)^{2\times 2}\cap \LL^2_\mathrm{sym}(\Omega) \big) \times \widehat U_h.
\end{align*}
This example gives numerical evidence that higher convergence rates are possible for one solution component by increasing the polynomial degree.

\begin{figure}
  \begin{center}
    \includegraphics{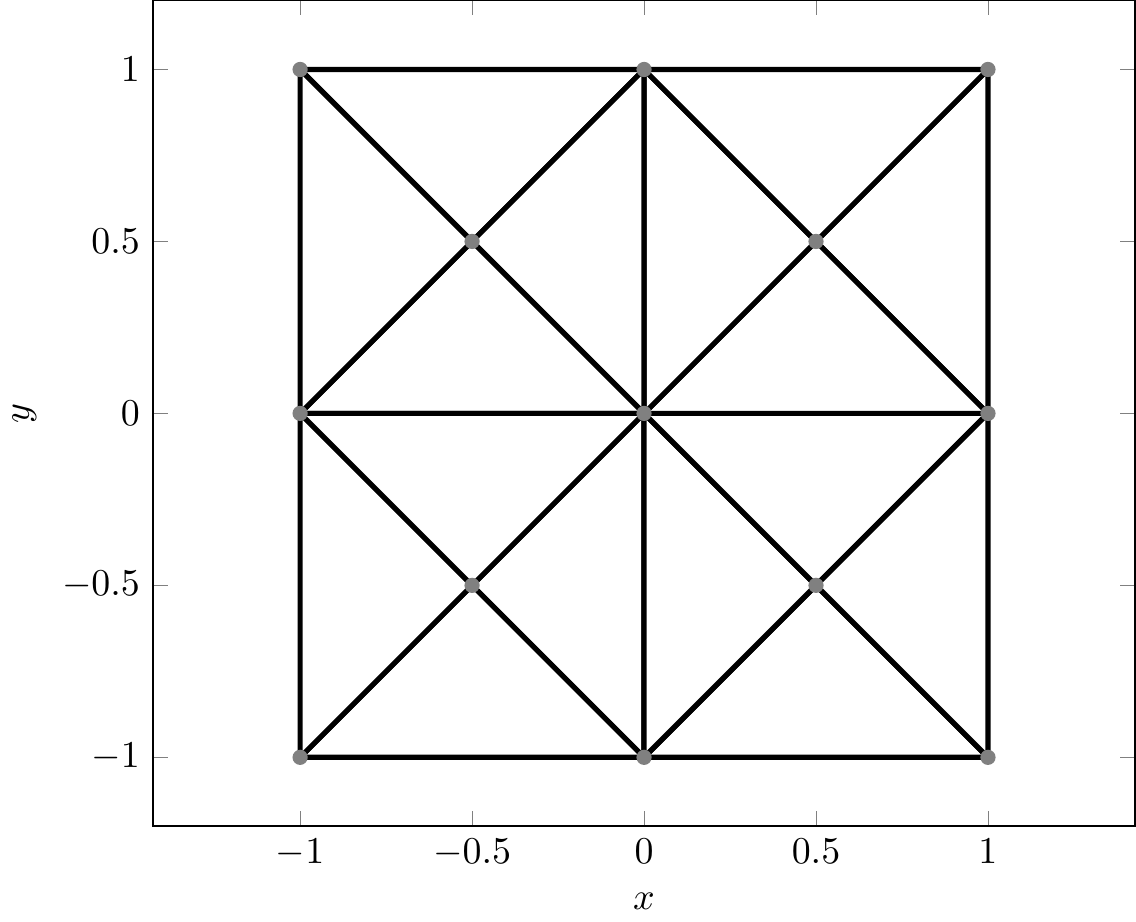}
  \end{center}
  \caption{Initial triangulation of $\Omega = (-1,1)^2$ with $16$ elements.}
  \label{fig:mesh}
\end{figure}

\subsection{Example with regular solution}\label{sec:ex1}
We consider the problem from~\cite[Section~6.1]{SmearsSueli13}, where
\begin{align*}
  A = A(x,y) = \begin{pmatrix}
    2 & \sign(xy) \\
    \sign(xy) & 2
  \end{pmatrix}
\end{align*}
and $f$ is chosen such that
\begin{align*}
  u(x,y) = \left( x e^{1-|x|}-x\right)\left(ye^{1-|y|}-y\right)
\end{align*}
is the exact solution of problem~\ref{eq:model}.
Note that $A$ satisfies the Cordes condition with $\varepsilon = 3/5$.
\begin{figure}
  \begin{center}
    \includegraphics{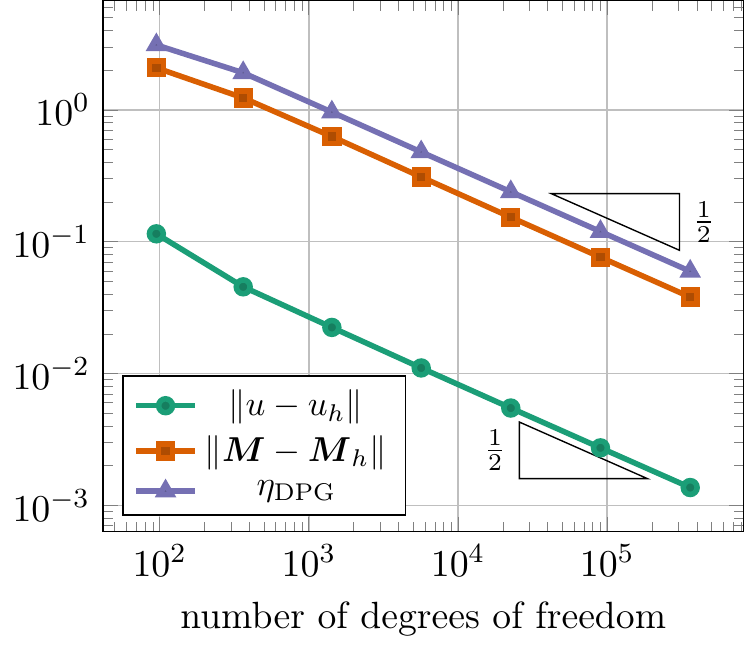}
    \includegraphics{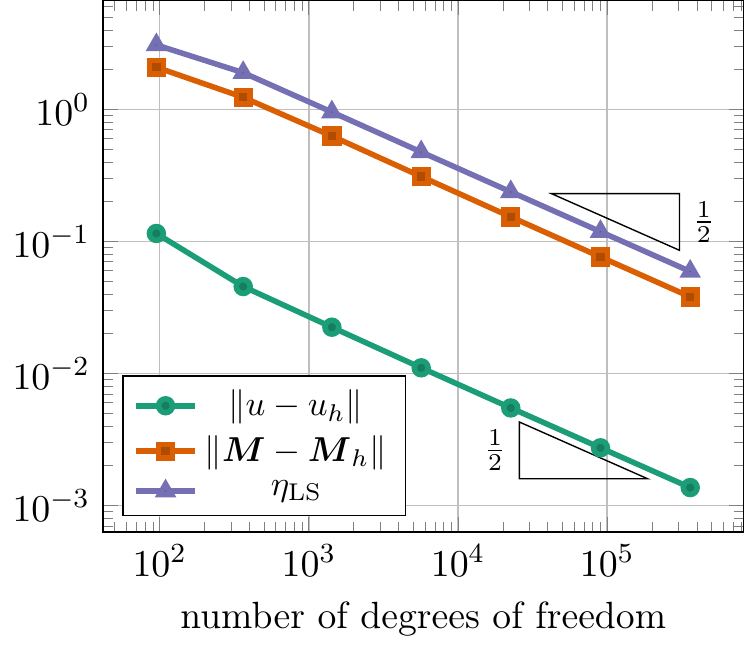}
  \end{center}
  \caption{Errors of field variables and estimators for the example from Section~\ref{sec:ex1}.}
  \label{fig:ex1}
\end{figure}
Figure~\ref{fig:ex1} shows the $L^2$ errors of the field variables compared to the error estimator.
The left plot shows the results for the method from Section~\ref{sec:dpgmethod} and the right plot the results for the method
from Section~\ref{sec:dpglsqmethod}.
Note that the discontinuities of $A$ are aligned with the initial mesh and that the coefficients of $A$ are constant on
each element. Since the two methods are equivalent we expect the same error curves which is also observed in
Figure~\ref{fig:ex1}.

\subsection{Example with known singular solution}\label{sec:exSing}
We use the same coefficient matrix as in Section~\ref{sec:ex1} and use the manufactured solution
\begin{align*}
  u(x,y) = (x^2+y^2)^{5/6}.
\end{align*}
The right-hand side data is then computed using~\eqref{eq:model}. We stress that $u$ does not satisfy the homogeneous
boundary conditions. We implemented the inhomogeneous boundary condition by lifting an approximation of $u|_\Gamma$,
see~\cite{DPGoverview}.

Note that $u\in H^{2+2/3-\delta}(\Omega)$ for all $\delta>0$. We therefore expect that uniform mesh-refinements lead to
suboptimal rates.
This can be observed in Figure~\ref{fig:exSing} which shows the error curves and the estimator for the method
from Section~\ref{sec:dpgmethod}.
Using an adaptive strategy the optimal rates are recovered, see again Figure~\ref{fig:exSing}.
A sequence of adaptive meshes created by the adaptive algorithm is visualized in Figure~\ref{fig:exSing:Mesh}. 
We observe strong refinements towards the ``singular'' node $(x,y) = (0,0)$.

\begin{figure}
  \begin{center}
    \includegraphics{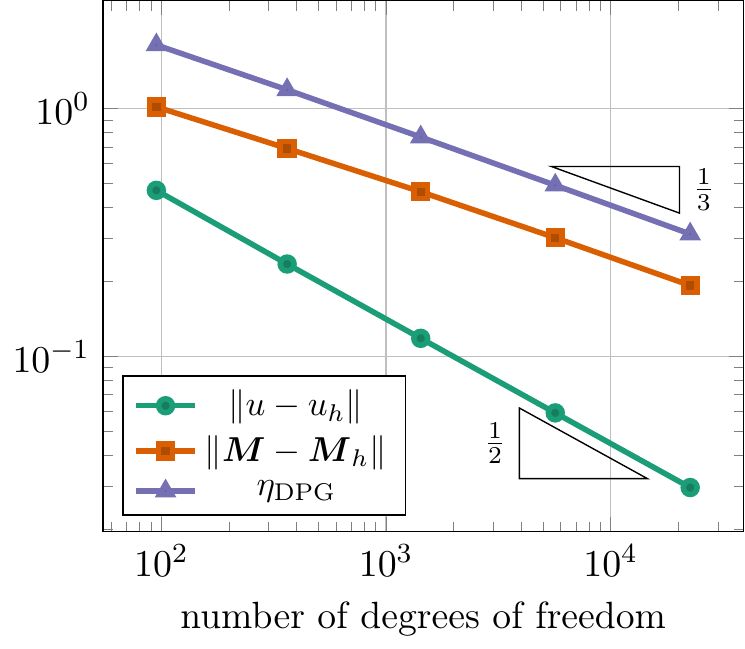}
    \includegraphics{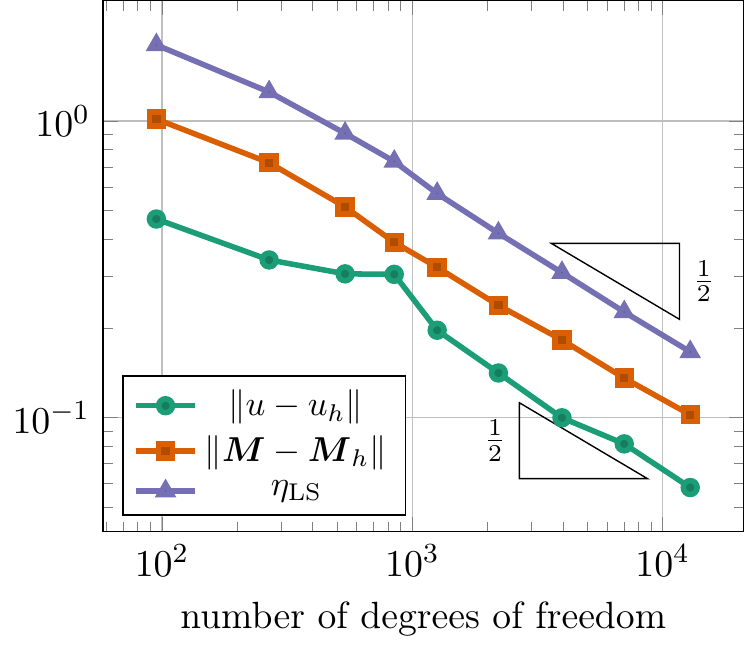}
  \end{center}
  \caption{$L^2$ field errors and error estimator for the example from Section~\ref{sec:exSing} The left plot shows the
  results on a sequence of uniform refinements and the right plot shows the results in the case of adaptive refinements.}
  \label{fig:exSing}
\end{figure}

\begin{figure}
  \begin{center}
    \includegraphics[width=0.35\textwidth]{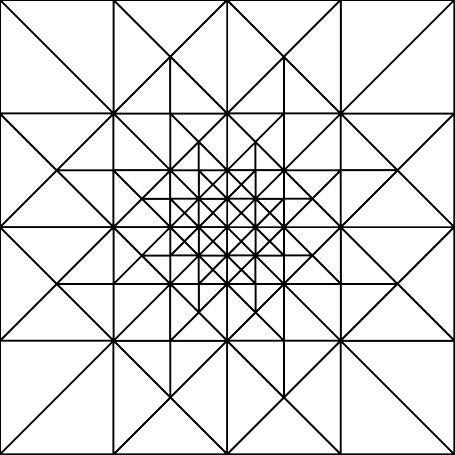}
    \includegraphics[width=0.35\textwidth]{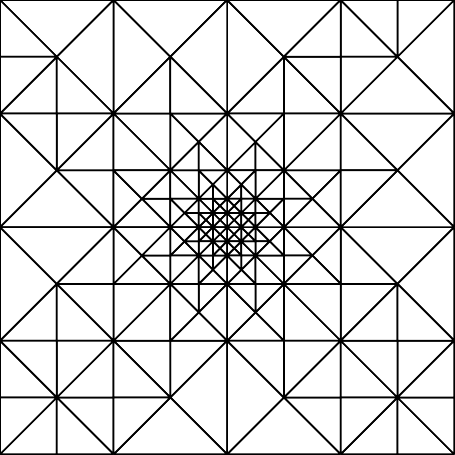}
    \includegraphics[width=0.35\textwidth]{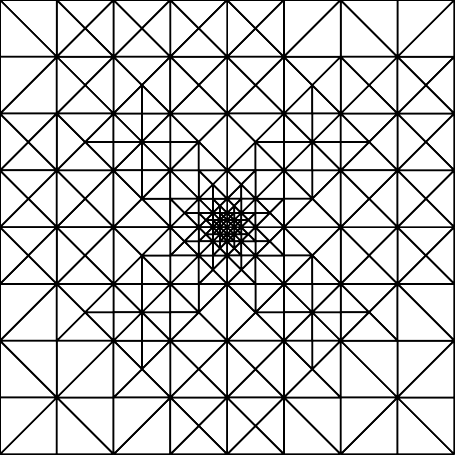}
    \includegraphics[width=0.35\textwidth]{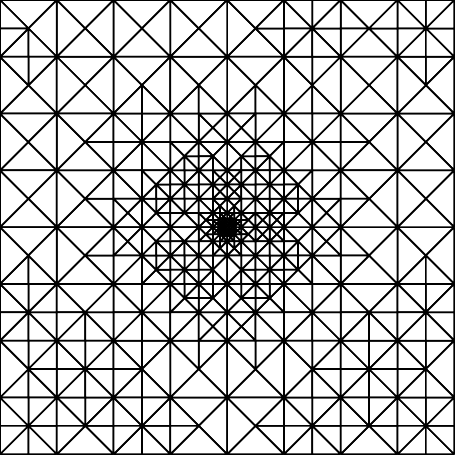}
  \end{center}
  \caption{Meshes generated by adaptive algorithm for the example from Section~\ref{sec:exSing}.}
  \label{fig:exSing:Mesh}
\end{figure}

\subsection{Example with unknown solution}\label{sec:ex2}
In this example we consider the solution to~\eqref{eq:model} with right-hand side $f=1$ and coefficient matrix
\begin{align*}
  A = A(x,y) = \begin{pmatrix}
  2 & g(x,y)\sign(xy) \\ g(x,y)\sign(xy) & 2 \end{pmatrix},
\end{align*}
where 
\begin{align*}
  g(x,y) = \begin{cases}
    1 & \text{if } 0\leq \sqrt{x^2+y^2} <1/3, \\
    -1& \text{if } 1/3 < \sqrt{x^2+y^2} <2/3, \\
    0 & \text{else}
  \end{cases}.
\end{align*}
In this case an exact solution is not known. 
As can be observed from Figure~\ref{fig:ex2} uniform mesh-refinement does not lead to the optimal orders of
convergence.
We therefore consider an adaptive algorithm where mesh-refinement is steered by the local mesh-indicators.
Figure~\ref{fig:ex2} shows the estimators for the method from
Section~\ref{sec:dpgmethod} and Section~\ref{sec:dpglsqmethod}. Note that the discontinuities of the coefficients are
not aligned with the mesh. Therefore, both methods are not equivalent (Theorem~\ref{thm:fortin} does not hold)
and we expect that different approximations are obtained which is also observed in Figure~\ref{fig:ex2}.

\begin{figure}
  \begin{center}
    \includegraphics{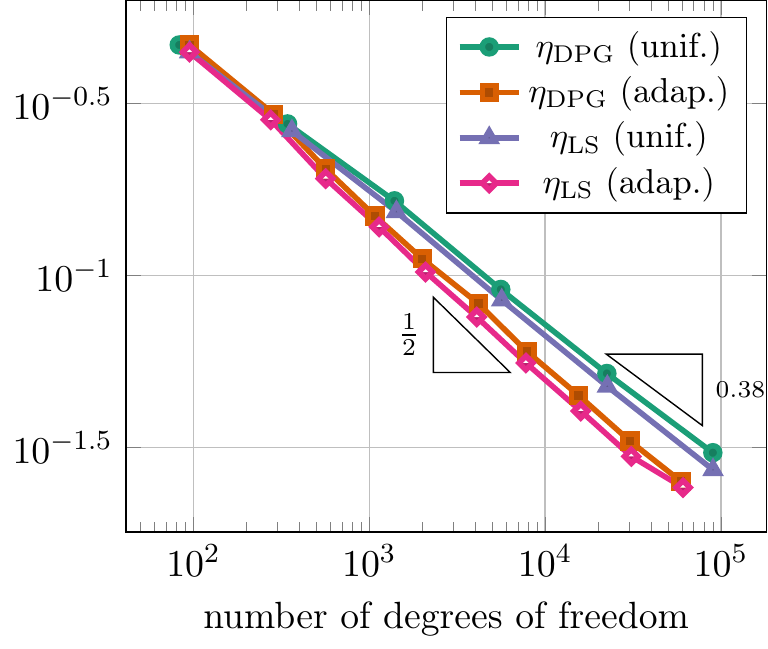}
  \end{center}
  \caption{Energy error for the example from Section~\ref{sec:ex2}.}
  \label{fig:ex2}
\end{figure}

\begin{figure}
  \begin{center}
    \includegraphics[width=0.3\textwidth]{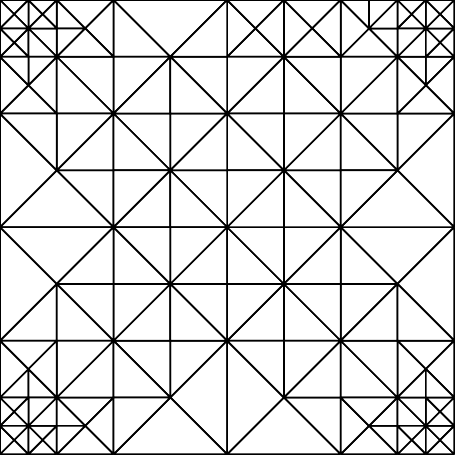}
    \includegraphics[width=0.45\textwidth]{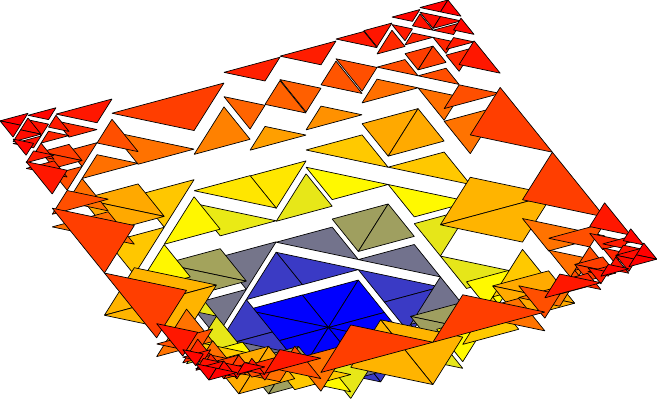}
    \includegraphics[width=0.3\textwidth]{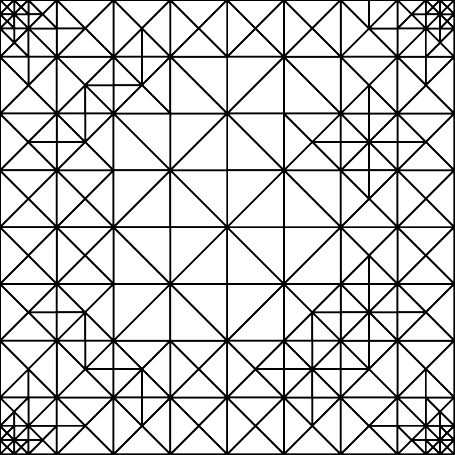}
    \includegraphics[width=0.45\textwidth]{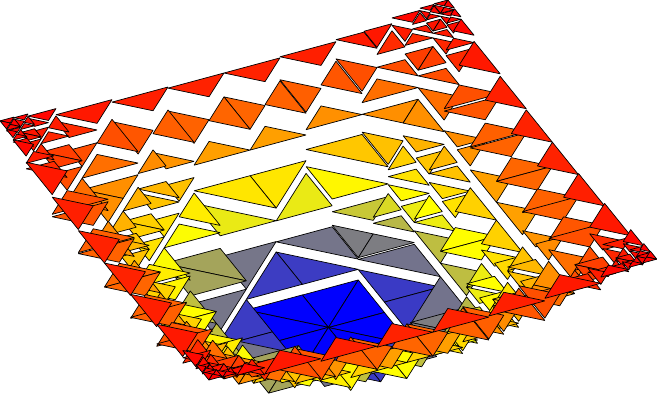}
    \includegraphics[width=0.3\textwidth]{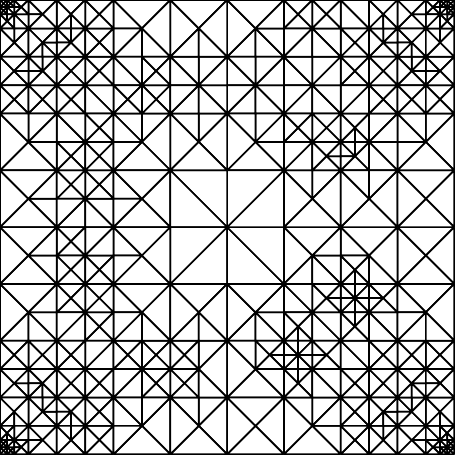}
    \includegraphics[width=0.45\textwidth]{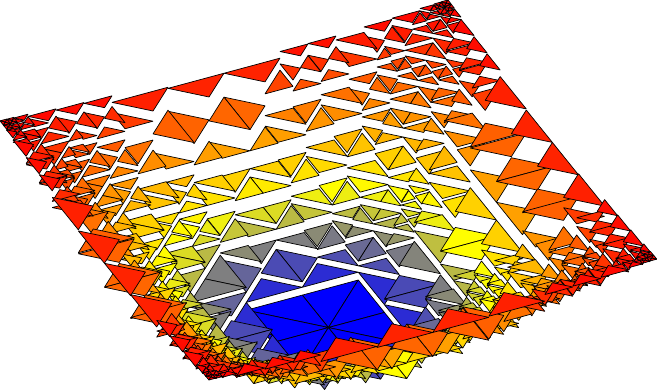}
    \includegraphics[width=0.3\textwidth]{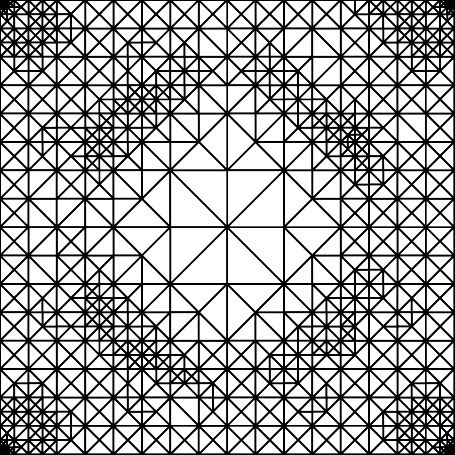}
    \includegraphics[width=0.45\textwidth]{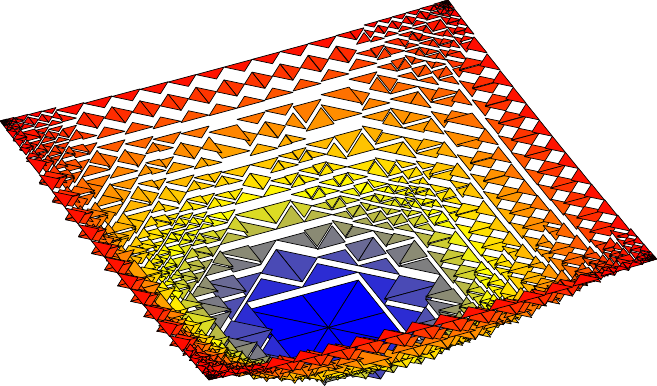}
  \end{center}
  \caption{Sequence of adaptively generated meshes and corresponding solution component $u_h$
  for the example from Section~\ref{sec:ex2}.}
  \label{fig:ex2:sol}
\end{figure}

\begin{figure}
  \begin{center}
    \includegraphics[width=0.3\textwidth]{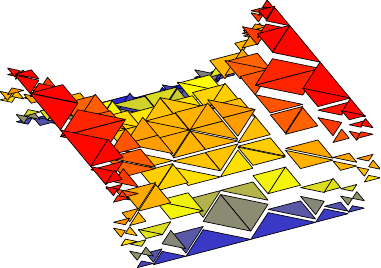}
    \includegraphics[width=0.3\textwidth]{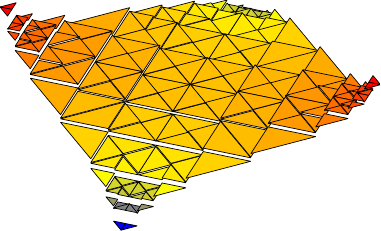}
    \includegraphics[width=0.3\textwidth]{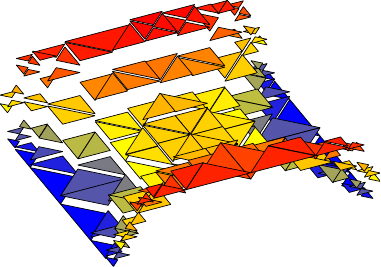}
    \includegraphics[width=0.3\textwidth]{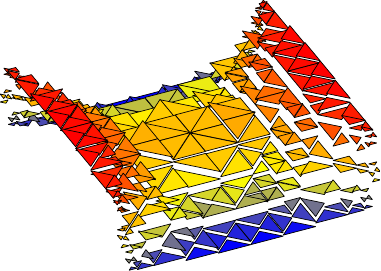}
    \includegraphics[width=0.3\textwidth]{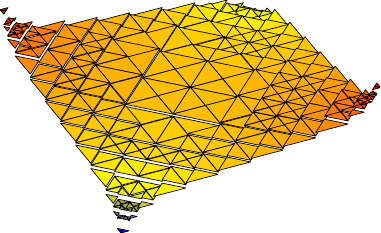}
    \includegraphics[width=0.3\textwidth]{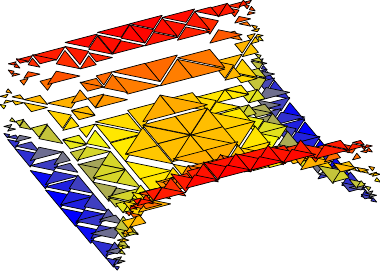}
    \includegraphics[width=0.3\textwidth]{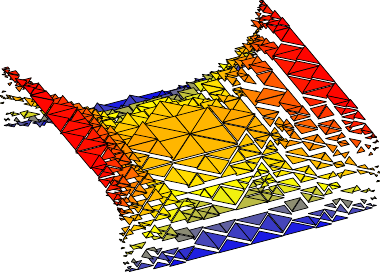}
    \includegraphics[width=0.3\textwidth]{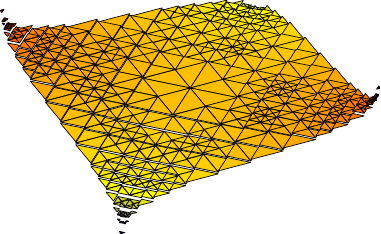}
    \includegraphics[width=0.3\textwidth]{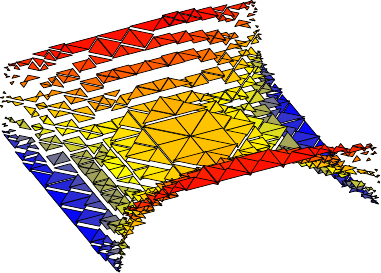}
    \includegraphics[width=0.3\textwidth]{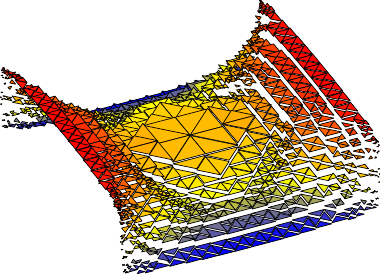}
    \includegraphics[width=0.3\textwidth]{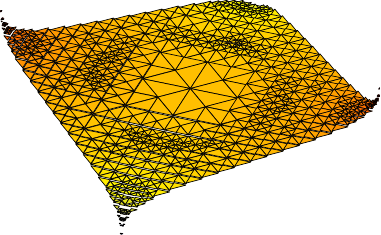}
    \includegraphics[width=0.3\textwidth]{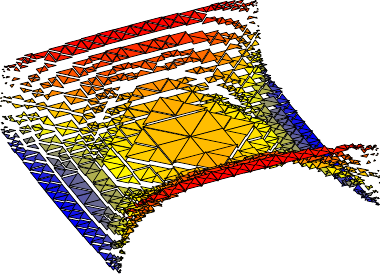}
  \end{center}
  \caption{Approximation of the Hessian $\MM = D^2 u$, i.e., $\MM_{h,11}$ (left), $\MM_{h,12}$ (middle), $\MM_{h,22}$
  (right) on a sequence of adaptively generated meshes for the problem from Section~\ref{sec:ex2}.}
  \label{fig:ex2:M}
\end{figure}

Figure~\ref{fig:ex2:sol} shows a sequence of meshes generated by the adaptive algorithm with corresponding solution
component $u_h$ (computed with the method from Section~\ref{sec:dpgmethod}). Figure~\ref{fig:ex2:M} visualizes the
corresponding approximations of the Hessian.

\subsection{Example with regular solution using augmented trial space}\label{sec:ex4}

\begin{figure}
  \begin{center}
    \includegraphics{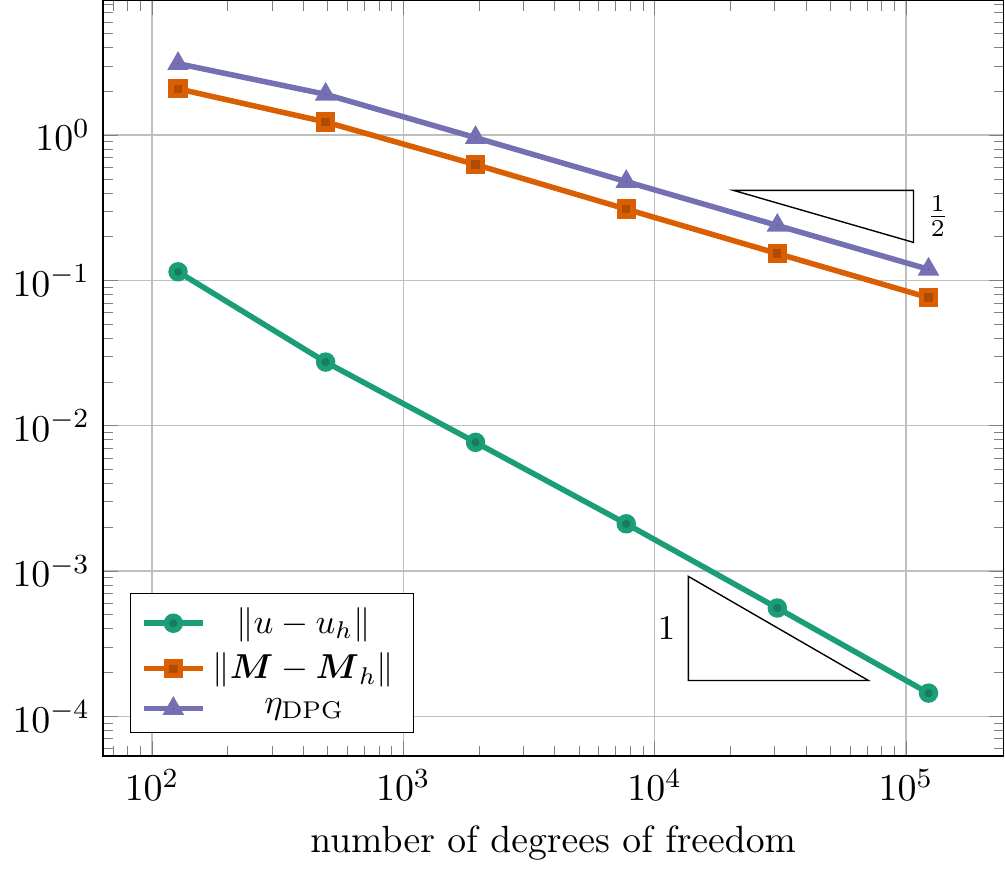}
  \end{center}
  \caption{Errors of field variables and estimators for the example from Section~\ref{sec:ex4}.}
  \label{fig:ex4}
\end{figure}

We consider the problem from Section~\ref{sec:ex1} but instead of seeking the approximation $\uu_h$ in $U_h$ we consider
the augmented trial space
\begin{align*}
  U_h^+ = \PP^1(\TT) \times \big( \PP^0(\TT)^{2\times 2}\cap \LL^2_\mathrm{sym}(\Omega) \big) \times \widehat U_h.
\end{align*}
The idea of using augmented trial spaces in DPG methods based on ultraweak formulations stems from the author's recent
works~\cite{SupConv1,SupConv2}.
There it was shown, under some standard regularity assumptions, that the use of augmented trial spaces leads to higher
convergence rates.
Figure~\ref{fig:ex4} gives numerical evidence that higher convergence rates are possible for the problem under
consideration.

\bibliographystyle{abbrv}
\bibliography{literature}

\end{document}